\documentclass[a4paper,10pt]{article}
\usepackage{float}
 \usepackage{latexstyle}
 \usepackage{bbm}
\usepackage[outdir=./]{epstopdf}
\usepackage{mathtools}
 \usepackage{framed}
 
\usepackage{fullpage}

\usepackage{amsmath}
\makeatletter
\renewcommand*\env@matrix[1][*\c@MaxMatrixCols c]{%
  \hskip -\arraycolsep
  \let\@ifnextchar\new@ifnextchar
  \array{#1}}
\makeatother

 \definecolor{green(pigment)}{rgb}{0.0, 0.65, 0.31}
\newcommand{\revision}[1]{{#1}}

 \definecolor{red(pigment)}{rgb}{0.65, 0.0, 0.31}
\newcommand{\rev}[1]{{#1}}

\title{Sampling the Fourier transform along radial lines}

\author{Charles Dossal \thanks{IMB, Universit\'e de Bordeaux (\texttt{charles.dossal@math.u-bordeaux.fr})}
  \and Vincent Duval\thanks{MOKAPLAN, INRIA Paris, CEREMADE,  Universit\'e  Paris-Dauphine (\texttt{vincent.duval@inria.fr})} 
\and Clarice Poon \thanks{
DAMTP, Centre for Mathematical Sciences, University of Cambridge
(\texttt{C.M.H.S.Poon@maths.cam.ac.uk})} \thanks{corresponding author}
}

\date{December, 2016; Revised May, 2017}
\begin{document}
\maketitle

\abstract{This article considers the use of total variation minimization for the recovery of a superposition of point sources from samples of its Fourier transform along radial lines. We present a numerical algorithm for the computation of solutions to this infinite dimensional problem. The theoretical results of this paper make precise the link between the sampling operator and the recoverability of the point sources.}

%
%
%

\section{Introduction}

Over the last decades, the use of regularizers for inverse problems has largely shifted from the use of Tikhonov regularization to $\ell^1$ minimization, following the pioneering works of Donoho and Tibshirani~\cite{chen1999atomi,tibshirani1996regre}. Using these approaches, one is generally led to search for a sparse signal on a predefined grid and to solve a finite-dimensional problem. Furthermore, the introduction of compressed sensing \cite{donohoCS,candes2006robust} has in particular triggered an intense amount of research into the notion of sparse recovery.
Although compressed sensing has largely been a finite dimensional theory, there has, in recent years, been several works studying the impact of sparsity for \textit{infinite-dimensional} inverse problems.
One aspect of this is to work in a gridless setting, replacing the discrete $\ell^1$ norm with its continuous counterpart, the total variation of Radon measures. This approach was proposed independently by several authors \cite{candes2014towards,Bredies-space-measures,de2012exact,bhaskar2013atomic}, and substantial mathematical insight was provided in \cite{candes2014towards}.

In this article, we continue this line of investigation and explore the use of total variation minimization for 
recovering the parameters of a superposition of point sources when restricted to sampling along radial lines of its Fourier transform. 
Our analysis reveals that the full total variation minimization problem can be solved by considering a sequence of univariate minimization problems. Utilising this observation, we describe an algorithm for solving the total variation minimization problem by solving a sequence of univariate semi-definite programs. Our approach is infinite dimensional in the sense that it allows for the recovery of the point sources without resorting to computations on a discrete grid. On the theoretical level, we provide sufficient conditions on the number of radial lines and the number of samples along these radial lines to guarantee exact reconstruction. Our main results show that in dimension $d$, one can recover the parameters of a superposition of $M$ point sources by sampling its Fourier transform along $d+1$ radial lines. Furthermore, the number of samples we require along each line is, up to $\log$ factors, linear with $M$.

\paragraph{Motivation} The problem of parameter estimation for a superpositions of point sources is rooted in applications such as astronomy, NMR (nuclear magnetic resonance) spectroscopy \cite{holland2011fast,wu2014situ,qu2011reconstruction,wenger2010compressed,kazimierczuk2011accelerated}  and microscopy \cite{puschmann2005super, studer2012compressive}. In these applications, the signal of interest can often be modelled as point sources and
limitations in the hardware mean that one is required to resolve fine details from low frequency data. 
Convex programming techniques for parameter estimation from low frequency measurements date back to the 1980's, with works in NMR spectroscopy \cite{newman1988maximization} and in seismic prospecting \cite{santosa1986linear}. Furthermore, since the publication  \cite{candes2014towards}, this subject has been a topic of intense research within the mathematical community  \cite{candes2013super, tang-mixtures2013,bendory2014robust,bendory2015exact,azais2015spike,duval2015exact,4Boyer_deCastro_Salmon16}. However, these works have focussed on the case where one samples the Fourier transform at Cartesian grid points. On the other hand, physical constraints can sometimes restrict observations to certain angular directions \cite{stewart2013high,stewart2015high}, and in the case of NMR spectroscopy, one is required to sample along continuous trajectories such as radial lines. In this article, we consider this problem under the additional constraint that one can only sample along radial lines in the Fourier domain.

\subsection{Problem statement}
\subsubsection{Acquisition model}
Let $d\in \bbN$ with $d\geq 2$ and let $\domain=\overline{B}(0,1/2)\subset \bbR^d$ be the centered closed ball with radius $1/2$. We denote by $\bbS^{d-1}$ the sphere embedded in $\bbR^d$ and by $\cM(\domain)$ the space of Radon measures with support in $\domain$. For $\theta\in \bbS^{d-1}$, let $\ell_\theta := \{ t \theta : t\in \bbR\}$ be the radial line directed by $\theta$ and $\projU$ the orthogonal projection onto $\ell_\theta$. Given $x\in \domain$, let $\delta_x$ denote the Dirac measure at $x$. The Fourier transform of $\mu\in\cM(\domain)$ at $\xi \in \bbR^d$ is defined by
$$
\cF \mu (\xi) = \int_{\bbR^d} e^{-i 2\pi \ip{\xi}{x}}\mu(\mathrm{d}x).
$$
Throughout this paper, given a finite set $E$, $\abs{E}$ denotes its cardinality.

In this article, we are interested in the recovery of a discrete measure $\mu_0 = \sum_{j=1}^M a_j \delta_{x_j}$ where $\{a_j\}_{j=1}^M\subset \bbC$ and $\{ x_j\}_{j=1}^M \subset \domain$, given  $\cfreq$ samples of its Fourier transform along $L$ radial lines.
More precisely, for $N\in \bbN$, let $\setfreq \subset \{ -N, \cdots, N\}$ with $\cfreq = \abs{\setfreq}$, and let $\setdir \subset \bbS^{d-1}$ with $\cdir = \abs{\setdir}$.   Then, the given observation is a vector  $y_0 := \Phi \mu_0 \in \bbC^{\cfreq\times \cdir}$, where 
\be{\label{eq:samp_op}
\Phi: \cM(\domain) \to \bbC^{\cfreq\times \cdir}, \qquad   \Phi \mu :=  \left(\cF \mu (k\theta)\right)_{(k,\theta)\in \setfreq\times\setdir}.
}
Note that the Fourier transform of $\mu_0$  at $\xi\in \bbR^d$ is equal to
$
\cF \mu_0(\xi) = \sum_{j=1}^M a_j e^{i2\pi \ip{\xi}{x_j}}.
$

\subsubsection{Total variation minimization}
Given the observation $y_0=\Phi\mu_0\in \bbC^{\cfreq\times \cdir}$ described above, in this article, we consider the solutions to the minimization problem
\be{\label{eq:tvmin}
  \min_{\mu \in \cM(\domain)} \br{\norm{\mu}_{TV} : \Phi \mu = y_0} \tag{$\cP_0(y_0)$}
}
where the total variation norm $\norm{\cdot}_{TV}$  is defined by
$$
\norm{\mu}_{TV} = \sup\br{\Re \mip{f}{\mu} : f\in \Cder{}(\domain), \norm{f}_\infty \leq 1},
$$
where $\Cder{}(\domain)=\Cder{}(\domain,\bbC)$ \revision{is the space of complex-valued continuous functions defined on $\domain$,} and $\mip{f}{\mu}:=\int_{\bbR^d}\bar{f}{\mathrm{d}\mu}$. 
In the case of a discrete measure $\mu_0 = \sum_{j=1}^M a_j \delta_{x_j}$, its total variation amounts to 
$  \norm{\mu_0}_{TV} = \sum_{j=1}^M\abs{a_j}.
$ So, in the discrete setting \eqref{eq:tvmin} is simply an $\ell^1$ minimization problem.

Our contributions in this article are two-fold: First, by considering the dual problem of \eqref{eq:tvmin}, we observe that in certain cases, one can decompose our multi-variate minimization problem into a sequence of univariate minimization problems. This naturally leads to semi-definite programming approach for the computation of minimizers of  \eqref{eq:tvmin}. Secondly, by analysing the conditions under which the proposed method successfully recovers solutions of \eqref{eq:tvmin}, we are led to a theoretical understanding of the solutions  of \eqref{eq:tvmin}. In particular, our main theorems will address the following questions:
\begin{enumerate}
\item[(Q1)] How many radial lines should one sample along?
\item[(Q2)] How many samples should one observe along each line?
\end{enumerate}

\paragraph{Overview}
In Section \ref{sec:sdp}, we derive our numerical algorithm, and in doing so, reveal some key properties which we would analyse for the derivation of the recoverability conditions for \eqref{eq:tvmin}. In Section \ref{sec:theorems}, we present the  main theoretical results, which address (Q1) and (Q2) on how one should sample to guarantee the uniqueness of solutions of \eqref{eq:tvmin}, and also when these solutions can be computed via our numerical algorithm. Numerical results are presented in Section \ref{sec:numerics}. \rev{Although this paper is devoted to the study of the noiseless problem \eqref{eq:tvmin}, we describe in Section \ref{sec:noise} how our algorithm can be extended to handle the noisy setting.}  Section \ref{sec:literature} gives an overview of other related works and also possible future directions for the present work. Finally, the proofs of our main results are presented in Section \ref{sec:proofs}.

\section{The numerical algorithm}\label{sec:sdp}

%

In this section, we describe  a gridless computational  approach to \eqref{eq:tvmin}. One of the key tools of this section was discovered in \cite{candes2014towards}, where the authors showed that a univariate total variation minimization problem (with Fourier sampling at grid points)  can be formulated as a semi-definite programming problem (SDP). The SDP approach allows for the computation of solutions to an infinite dimensional optimization problem without resorting to computations on a discrete grid. Utilising this link with SDP, we describe how, in certain cases, \eqref{eq:tvmin} can be solved via a sequence of univariate SDP's. We begin this section by recalling some facts about the dual formulation of \eqref{eq:tvmin}.

\subsection{The dual formulation} \label{sec:dual_form}

Given $y_0=\Phi \mu_0$ with $\mu_0\in\cM(\domain)$, the dual problem\footnote{For all aspects of convex optimization, duality, subdifferential\ldots{} we shall regard a complex valued measure $m$ as a real vector-valued measure. We rely on the duality between $\Cder{}(\domain,\bbR^2)$ and $\cM(\domain,\bbR^2)$ and we apply the results in~\cite{EkelandTemam} for real locally convex vector spaces.} of \eqref{eq:tvmin} is
\be{\label{eq:dual_tv}
  \sup \left\{\Re\ip{q}{y_0} : q\in \bbC^{\revision{\cfreq\times \cdir}} \quad \norm{\Phi^*q}_\infty \leq 1 \right\}, \tag{$\cD_0(y_0)$}
}
where given $q,y\in\bbC^{T\times L}$, $\ip{q}{y}:=\sum_{(k, \theta)\in \setfreq\times \setdir} {y}_{k,\theta}{\bar{q}_{k,\theta}}$.
Observe that when endowed with the weak-* topology, the dual of the space $\cM(\domain)$ is $\Cder{}(\domain)$, and 
the adjoint operator $\Phi^*$ of $\Phi$ is given by 
\begin{equation}\label{eq:phistar}
\Phi^*:   \bbC^{\cfreq\times \cdir} \to \Cder{}(\domain), \quad q\in \bbC^{\cfreq\times \cdir}  \mapsto \sum_{(k, \theta)\in \setfreq\times \setdir}  q_{\theta, k} e^{i2\pi k\ip{\theta}{\cdot}}.
\end{equation}
 Moreover, one can prove the  existence of solutions to both ~\eqref{eq:tvmin} and~\eqref{eq:dual_tv}, and that strong duality holds, that is $\eqref{eq:tvmin}=\eqref{eq:dual_tv}$. The primal and dual problems are related such that   $\mu\in \cM(\domain)$ solves \eqref{eq:tvmin} and $q\in  \bbC^{\cfreq\times \cdir}$ solves \eqref{eq:dual_tv} if and only if
\be{\label{eq:optcdt}%
\Phi \mu = y_0, \quad \text{and}
\quad 
\Phi^* q \in \br{f\in \Cder{}(X): \norm{f}_{\infty}\leq 1, \quad \ip{f}{\mu} = \norm{\mu}_{TV}}.
}
For a discrete measure $\mu_0=\sum_{j=1}^M a_j \delta_{x_j}$,  if there exists  $q \in \bbC^{\cfreq\times \cdir}$ such that
 \be{\label{toshow_unique}
  \Phi^*q(x_j) = \sgn(a_j):=\frac{a_j}{\abs{a_j}}, \quad \forall j\in \{1, \ldots, M\}, \qquad  \norm{\Phi^*q  }_\infty \leq 1, 
}
and the extremal points of $\Phi^* q$ form a finite set $E:=\br{x: \abs{\Phi^* q(x) }= 1}$ such that
the map $b \in \bbC^{\abs{E}} \mapsto \Phi\left(\sum_{x\in E}^{} b_x \delta_{x}\right)$  is injective,
then $\mu_0$ is the unique solution of \eqref{eq:tvmin}. This result is essentially proved in ~\cite[Lem. 1.1]{de2012exact} where $\sgn(a_j)$ are real numbers, note however, that a similar result for complex numbers is also proved in \cite[Prop. A.1]{candes2014towards}, and our conditions in fact imply the conditions of \cite[Prop. A.1]{candes2014towards}.
Since finding a vector $q$ which satisfies~\eqref{toshow_unique} guarantees that $\mu_0$ is a solution, we shall call $\Phi^*q$ a dual certificate.

\subsection{Splitting the dual problem}

For each $\theta\in\setdir$, let $y_\theta := \left(\cF \mu_0(k\theta )\right)_{ k\in\setfreq}$, and let $\subsetdir\subseteq \Theta$ be any subset of cardinality $\cdir'\leq \cdir$. Instead of \eqref{eq:dual_tv}, let us consider the following optimization problem:
\begin{align}
 \sup&\enscond{\frac{1}{L'}\sum_{\theta\in \subsetdir}\Re\ip{c_\theta}{y_\theta}}{\forall \theta\in\subsetdir,\ c_\theta\in \bbC^{\cfreq}, \mbox{ and } \sup_{x\in \domain}\abs{\sum_{k\in \setfreq}  c_{\theta, k} e^{i2\pi k\ip{\theta}{x}}}\leq 1}\nonumber\\
                         &= \frac{1}{L'}\sum_{\theta\in \subsetdir}\sup\enscond{\Re\ip{c_\theta}{y_\theta}}{c_\theta\in \bbC^{\cfreq}, \mbox{ and }\sup_{x\in \domain}\abs{\sum_{k\in \setfreq}  c_{\theta, k} e^{i2\pi k\ip{\theta}{x}}}\leq 1}.\label{eq:dualsplit2} \tag{$\tilde{\cD}_0(y_0)$}
\end{align}
\revision{Given any family $(c_\theta)_{\theta\in\subsetdir}$ admissible for~\eqref{eq:dualsplit2}, we may construct $q\in \bbC^{\cdir\times\cfreq}$ with
  \begin{align}\label{eq:defq}
  q_{\theta,k}= \begin{cases}
    \frac{c_{\theta, k}}{L'} &\ \mbox{for $\theta\in\subsetdir$},\\
    0 & \ \mbox{for $\theta\in\setdir\setminus\subsetdir$}.
  \end{cases}
\end{align}
and we see that $\norm{\Phi^*q}_{\infty}\leq 1$ and $q$ is admissible for~\eqref{eq:dual_tv}, with $\Re\ip{q}{y_0}=\frac{1}{L'}\sum_{\theta\in \subsetdir}\Re\ip{c_\theta}{y_\theta}$. As a result $\eqref{eq:dualsplit2}\leq \eqref{eq:dual_tv}\leq \eqref{eq:tvmin}$.

Now, suppose that  for each $\theta\in \subsetdir$, there exists $c_\theta \in \bbC^T$ such that
\be{\label{eq:assump1}
 p_\theta: t \mapsto \sum_{k\in \setfreq}  c_{\theta, k} e^{i2\pi k t } \mbox{ satisfies } 
 \norm{p_\theta}_\infty\leq 1, \quad \mbox{ and } p_\theta(\ip{\theta}{x_j}) = \sgn(a_j), \quad  \forall j\in\{1, \ldots, M\}.
}
Then, $q$ defined as in~\eqref{eq:defq} is admissible for~\eqref{eq:dual_tv} and satisfies $(\Phi^*q)(x_j) = \sgn(a_j)$ for all $j\in\{1,\ldots, M\}$.

As a result, 
\begin{align*}
  \frac{1}{L'}\sum_{\theta\in \subsetdir}\Re\ip{c_\theta}{y_\theta}=\Re\ip{q}{y_0}=\ip{\Phi^*q}{\mu_0}=\sum_{j=1}^M\abs{a_j}=\norm{\mu_0}_{TV},
\end{align*}
and we deduce that $\eqref{eq:dualsplit2}=\eqref{eq:dual_tv}=\eqref{eq:tvmin}$ and that $\mu_0$ is optimal for~\eqref{eq:tvmin}.
Incidentally, each $c_\theta$ maximizes the corresponding summand in \eqref{eq:dualsplit2}.
}

To see why this observation is useful, we recall an observation from \cite{candes2014towards} concerning the solution to the dual problem in the univariate case:
A univariate trigonometric polynomial satisfies
$\norm{\sum_{j\in\setfreq} c_{j} e^{i2\pi j \cdot}}_\infty\leq 1$ with $\setfreq\subset \br{-N, \ldots, N}$ if and only if there exists a Hermitian matrix $Q\in \bbC^{ N\times  N}$ such that
$$
\begin{bmatrix}
Q &c\\
c^* & 1
\end{bmatrix} \succeq 0, \qquad \sum_{i=1}^{N-j}Q_{i,i+j}=\begin{cases}
1 &j=0\\
0 &j=1,2,\ldots, N-1
\end{cases}, \qquad c_{\setfreq^c} = 0,
$$
where $c_{\setfreq^c}$ is the restriction of $c$ to coefficients indexed by $\setfreq^c$. So, the dual problem
\begin{align*}
&\sup_{ c\in\bbC^N}\br{ \Re \ip{c}{y}: p = \sum_{j\in\setfreq} c_{j} e^{i2\pi j \cdot}, \norm{p}_\infty\leq 1}
\end{align*}
is equivalent to
\begin{align*}
\sup_{Q\in \bbC^{N\times N}, ~c\in\bbC^N} \br{\Re \ip{c}{y}: 
\begin{bmatrix}
Q &c\\
c^* & 1
\end{bmatrix} \succeq 0, \quad \sum_{i=1}^{N-j}Q_{i,i+j}=\begin{cases}
1 &j=0\\
0 &j=1,\ldots, N-1
\end{cases}, \quad c_{\setfreq^c} = 0
}.
\end{align*}
Crucially, the latter equation can be solved using semi-definite programming (SDP).

From this observation, it is clear that since  \eqref{eq:dualsplit2} is formulated in terms of $\cdir'$ univariate trigonometric polynomials,   \eqref{eq:dualsplit2}  can be solved via  a sequence of univariate SDPs.
More precisely, each  summand in \eqref{eq:dualsplit2} can be rewritten as
\eas{
  \sup_{Q\in \bbC^{N\times N},~ c\in \bbC^N}\left\{ \Re  \ip{c}{y_\theta} : \begin{bmatrix}
Q &c\\
c^* & 1
\end{bmatrix} \succeq 0, \qquad \sum_{i=1}^{N-j}Q_{i,i+j}=\begin{cases}
1 &j=0\\
0 &j=1,\ldots, N-1
\end{cases}, \quad c_{\setfreq^c} = 0
\right\}.
}
Under the assumption \eqref{eq:assump1}, for each $\theta\in \subsetdir$, $( \ip{x_j}{\theta})_{j=1}^M$ is contained in the extremal points of the trigonometric polynomial
 $
p_\theta$, which we denote by $\cT_\theta$.
In particular, for all $x\in \{x_j:j=1,\ldots,M\}$ and  all $\theta\in \subsetdir$, there exists $t\in \cT_\theta$ such that $\ip{x}{\theta} = t$, i.e. $x\in t\theta +\ell_\theta^\perp$. So, 
\be{\label{eq:Delta_bar}
\supsat_\subsetdir  := \bigcap_{\theta\in \subsetdir} \bigcup_{t \in \cT_\theta} \left( t \theta + \ell_\theta^\perp  \right) \supset \{x_j:j=1,\ldots, M\}.
}
In the following, we write $\supsat_{\subsetdir} =\br{\supx_j:~ j=1,\ldots, \supM}$ provided $\supsat_{\subsetdir}$ is finite.
In that case, assuming that the operator 
\be{\label{eq:A2}
A_{\Theta,\Gamma,\supsat_{\subsetdir}}:  a \in \bbC^{\supM}\mapsto  \left(\sum_{j=1}^{\supM} a_{j} e^{-i2\pi k\ip{\theta}{\supx_j}}\right)_{\theta\in\setdir, k\in\setfreq}  \in   \bbC^{T\times L} \qquad \text{is injective},
}
 the solution $(\supa_j)_{j=1}^{\supM}$ to the linear system $A_{\Theta,\Gamma,\supsat_{\subsetdir}} \supa = y_0$ will satisfy
$
\sum_{j=1}^{\supM} \supa_j \delta_{\supx_j} =\mu_0.
$

This discussion suggests the following  reconstruction procedure:
\begin{framed}
\textbf{Algorithm} \algname:

Let $L'\leq L$ be some parameter chosen by the user.
\begin{enumerate}
\item For each $\theta \in \Theta$, via SDP, compute the maximizer $c_\theta$ of
\eas{
\sup\left\{ \Re  \ip{y_\theta}{c} : \begin{bmatrix}
Q &c\\
c^* & 1
\end{bmatrix} \succeq 0, \qquad \sum_{i=1}^{N-j}Q_{i,i+j}=\begin{cases}
1 &j=0\\
0 &j=1,\ldots, N-1
\end{cases}, \quad c_{\setfreq^c} = 0
\right\}.}
Let $\cT_\theta$ be the extremal values of the trigonometric polynomial
$$
p_\theta = \sum_{k\in\Gamma} c_{\theta,k} e^{\revision{2\imath}\pi k\cdot}.
$$
The values are  found by constructing the companion matrix of the associated algebraic polynomial\footnotemark{}  and finding its eigenvalues with absolute value within a threshold of 1 (in our experiments, we choose a threshold of $10^{-4}$). 
\item For each subset $\subsetdir\subset \Theta$ of cardinality $L'$, let $\supsat_\subsetdir$ be as defined in \eqref{eq:Delta_bar}, and $A_{\Theta,\Gamma,\supsat_{\subsetdir}}$ be as defined in \eqref{eq:A2}. If $\supsat_{\subsetdir}$ is a discrete set $\{\supx_j: j=1,\ldots,\supM\}$ and $A_{\Theta,\Gamma,\supsat_{\subsetdir}}$ is injective, then recover $a_\subsetdir$ as the solution to the  linear system $A_{\Theta,\Gamma,\supsat_{\subsetdir}} a_\subsetdir = y_0$, and let 
$$
\mu_\subsetdir = \sum_{j=1}^{\supM} a_{\subsetdir,j} \delta_{\supx_j}.
$$ 
If  $\sgn(a_{\Theta'})_j = \frac{1}{\abs{\Theta'}} \sum_{\theta\in\Theta'} p_\theta(\ip{x_j}{\theta})$ for all $j\in \br{1,\ldots, \tilde M}$, then return $\mu_\subsetdir$.
\end{enumerate}

\end{framed}
\footnotetext{As in~\cite{candes2014towards}, we associate to each trigonometric polynomial $t\mapsto \sum_{-N\leq k\leq N}c_ke^{2i\pi kt}$ the polynomial $1-X^{2N}\sum_{-N\leq k,j\leq N}c_k\overline{c}_jX^{k-j}$. Its roots on the unit circle yield the points where the trigonometric polynomial reaches $1$ in modulus.
}
\revision{If the algorithm finishes step 2  without having returned any measure, it means that~\eqref{eq:tvmin} cannot be solved using the proposed splitting approach (with this choice of $L'$). The rest of this paper is devoted to showing that the splitting approach succeeds in many practical cases.}
Note that when $L'<L$, the above algorithm may return several measures, however, the last assertion in Step 2 above ensures that each one of these measures is a minimizer of \eqref{eq:tvmin}. Furthermore, by collecting the assumptions in the discussion above, the algorithm \algname  ~returns precisely $\mu_0$  provided that the following two conditions are satisfied:
\begin{itemize}
\item[(A1)] There exists $\subsetdir\subset\Theta$ of cardinality $L'$ such that for each $\theta\in \subsetdir$, there exists $~c_\theta \in \bbC^T$ such that the trigonometric polynomial $p_\theta = \Phiod^* c_\theta$ satisfies
\bes{
\norm{p_\theta}_\infty\leq 1, \quad p_\theta(\ip{\theta}{x_j}) = \sgn(a_j), \quad \forall j = 1,\ldots, M.
}
\item[(A2)] For the set $\subsetdir$ defined in (A1), $\supsat_{\subsetdir}$ defined in \eqref{eq:Delta_bar}, $A_{\Theta,\Gamma,\supsat_{\subsetdir}}$ is injective.
\end{itemize}

\revision{As mentioned in Section \ref{sec:dual_form}, a discrete measure $\mu_0 = \sum_{j=1}^M a_j \delta_{x_j}$ with $\Phi \mu = y_0$ is a solution to \eqref{eq:tvmin} if there exists a dual certificate $q \in \bbC^{\cfreq\times \cdir}$ such that $\Phi^*q(x_j) = \sgn(a_j)$ and $\norm{\Phi^*q}_\infty\leq 1$. Since Algorithm \algname  ~describes how to construct such dual certificate}, we have in fact derived the following statement:
\begin{proposition}\label{prop:split}
If (A1) and (A2) hold, then $\mu_0$ is the unique solution to \eqref{eq:tvmin} and is recovered by Algorithm \algname. 
\end{proposition}

\begin{remark}
Note that by choosing $L'<L$, it is easier for the conditions (A1) and (A2) to be satisfied, however, there is an increased computational cost in Step 2. Although we do not theoretically study the case $L'<L$, we present a numerical example in Section~\ref{sec:numerics} to demonstrate one situation where it may be advantageous to let $L'<L$.
\end{remark}

In the next section, we present some results describing how one may choose $\Theta$ and $\Gamma$ such that (A1) and (A2) are satisfied with $L' = L=d+1$. 

\section{Theoretical results}\label{sec:theorems}

Before presenting our main results which provide conditions under which (A1) and (A2) are satisfied,
let us recall some of the existing results on total variation minimization.

\paragraph{Existing results on total variation minimization}

A significant mathematical breakthrough in the understanding of the total variation minimization problem is \cite{candes2014towards}, where the authors  derived precise conditions under which one can exactly solve the parameter estimation problem for $M$ point sources.  For the purpose of analysing  the problem of  recovering point sources from samples of the Fourier transform at grid points, the key notion introduced in \cite{candes2014towards} is  the minimum separation distance.

\begin{definition}
Let $\bbT$ denote the one-dimensional torus.
Given any discrete set $\Delta \subset \bbT$, let the minimum separation distance of $\Delta$ be defined by
$$
\mdist(\Delta) = \min_{t,t'\in \Delta, t\neq t'}\dt{t-t'}.
$$
where $\dt{t-t'}$ is the canonical distance on the torus between $t$ and $t'$.
\end{definition}
 Their main result shows that one is guaranteed exact recovery provided that the sampling range is inversely proportional to the minimum separation distance between the positions of the point sources. 
 More precisely, it suffices to sample Fourier coefficients with frequency no greater than $f_c$ provided that the positions of the point sources are separated by a distance of at least $\ord{1/f_c}$.  Explicit constants are given in the one and two variate setting, see \cite{candes2014towards} for further details and~\cite{fernandez2016super} for improved constants. Moreover, on the practical side, the authors presented a numerical algorithm based on semi-definite programming which allows for the computation of the minimizer without resorting to discrete grids.

 The theoretical result of \cite{candes2014towards} is extended to the probabilistic framework in \cite{tang2013compressed} where it is shown that one can in fact guarantee exact recovery with high probability by subsampling at random the Fourier coefficients with frequency no greater than $f_c$, at a rate which is up to log factors linear with sparsity. Several other variants have been proposed, replacing the Dirichlet kernel with more general kernels such as the Cauchy or Gaussian kernel~\cite{tang-mixtures2013,bendory2014robust}, or extending the framework to spherical domains~\cite{bendory2015exact}. The robustness to noise of such methods is investigated in~\cite{candes2013super,azais2015spike,duval2015exact,4Boyer_deCastro_Salmon16}.

In the remainder of this section, we present our main results, which can be seen as the analogy of the results of \cite{candes2014towards} and \cite{tang2013compressed} in the case of sampling the Fourier transform along radial lines.

\subsection{Main result I}

The key notion in our setting is the minimum separation distance between the \textit{projected} positions. In particular, if one is restricted to sampling along the directions specified by some set $\Theta\subset \bbS^{d-1}$, then along each line, one should observe the Fourier samples up to frequency $N$, which is  inversely proportional to
$
\inf_{\theta\in \Theta} \mdist(\setptheta)$ where $
 \setptheta = \{ \ip{\theta}{x_j}\}_{j=1}^M.
$
Now, as we show in Lemma \ref{lem:holdae}, there are in fact only a finite number of directions $\Theta$ for which this value is zero, so it is finite for almost every choice of $\Theta$. However, it may become arbitrarily small depending on the choice of the set $\Theta$. Understanding this minimum separation distance forms a substantial part of the analysis in this work, and in our first result, we make precise the dependence between the minimum separation distance of the projected positions of the underlying point sources and the sampling range. Furthermore, by allowing for random sampling along radial lines, provided that we sample from a `good' range of angles based on some a-priori knowledge on the signal, up to log factors, one can guarantee exact recovery with high probability using total variation minimization from $\ord{(d+1)M}$ samples.

\begin{theorem}\label{thm:rand_angles}
Let $\mu_0 = \sum_{j=1}^M  a_j \delta_{x_j}$ where $\{x_j\}_{j=1}^M \subset \domain$ consists of fixed distinct points.
Let $S\subset \bbS^{d-1}$ be a set of non-zero measure  such that
$$
\nu_{\min}:= \inf_{\theta\in S} \mdist(\setptheta) >0,
$$
where $\setptheta = \{ \ip{\theta}{x_j}\}_{j=1}^M$.
Let  $\setdir$ be a set of $d+1$ distinct elements drawn uniformly at random from $S$ and let $N = \lceil 2/\nu_{\min} \rceil$. Then, the following holds:
\begin{enumerate}
\item  If $\setfreq = \{-N,\ldots, N\}$, then $\mu_0$ is the unique solution to \eqref{eq:tvmin} and it can be recovered by Algorithm~\algname.
\item If $\setfreq$ consists of $m$ indices drawn uniformly at random from $\{-N,\ldots, N\}$, where 
$$
m\gtrsim \max \{\log^2(N/\delta), M\log(M/\delta)\log(N/\delta)  \},
$$
 and  $\{\sgn(a_j)\}_{j=1}^M$ are  drawn i.i.d. from the uniform
distribution on the complex unit circle, 
then with probability exceeding $1-(d+1)\delta$, 
 $\mu_0$ is the unique solution to \eqref{eq:tvmin} and it can be recovered by Algorithm~\algname.
\end{enumerate}

\end{theorem}

\begin{remark}
As in \cite{tang2013compressed}, we need to assume in 2 that $\{\sgn(a_j)\}_{j=1}^M$ are  drawn i.i.d. from the uniform distribution on the complex unit circle. It is likely that this is simply an artefact of proof techniques rather than an actual requirement.

\end{remark}

\subsection{Main result II: Recovery with fixed sampling directions}

The previous result considered the recovery of some fixed signal from Fourier samples along a random selection of radial lines. In this subsection, we consider the case where the radial lines which we sample along are fixed, but the positions of the point sources are distributed uniformly in space.

The following result describes when the algorithm presented in this article succeeds with $L=L'=d+1$. Under this setting, the sampling range along each line is $\ord{M^2}$. Furthermore, for exact recovery, it suffices to randomly draw $\ord{M\log^2 M}$  Fourier samples along the $d+1$ prescribed radial lines.  

\begin{theorem}\label{thm2}
Let $\mu_0 = \sum_{j=1}^M a_j \delta_{x_j}$ where  the positions $(x_j)_{j=1}^M\in \domain$ are i.i.d. random variables with uniform law in $\domain$.  Let $\delta>0$ and let $N \geq \frac{4(d+1)}{\delta \sqrt{\pi(2d-1)}}M(M-1)$ and let $\Theta \subset \bbS^{d-1}$ be a set of $d+1$ fixed distinct angles, of which any subset of cardinality $d$ is linearly independent.
\begin{enumerate}
  \item  If $\setfreq=\{-N, \ldots,N\}$, then with probability at least $1-(d+1)\delta$,  $\mu_0$ is the unique solution of \eqref{eq:tvmin} and can be recovered by Algorithm~\algname.

  \item If $\setfreq$ consists of $m$ indices chosen uniformly at random from $\br{-N,\ldots, N}$ with
$$
m\gtrsim   M\log^2(M/\delta),
$$
and $\{\sgn(a_j)\}_{j=1}^M$ are  drawn i.i.d. from the uniform
distribution on the complex unit circle, then, with probability at least 
 $ 1-(d+1)\delta$, 
 $\mu_0$ is the unique solution of \eqref{eq:tvmin} and can be recovered by Algorithm~\algname.

\end{enumerate}
\end{theorem}

\paragraph{Comparison with Cartesian grid sampling}
Let us comment further on Theorem \ref{thm2} in the bivariate case.
In \cite{candes2014towards}, it was proved that if  $\mu_0 = \sum_{j=1}^M a_j \delta_{x_j}$ where $\Delta := \br{x_j: j=1,\ldots, M} \subset \bbT^2$ and $\mdist(\Delta) \geq 2.38/N$, then $\mu_0$ can be exactly recovered by solving \eqref{eq:tvmin} with the following sampling operator:
$$
\Phi_{\mathrm{Cartesian}}: \mu \mapsto  \br{\cF \mu (k) : k\in\bbZ^2: \abs{k}_\infty\leq N}.
$$
Now, if the positions in $\Delta$ are chosen uniformly at random, one can show (see Lemmas \ref{lemdistance} and \ref{lem:distance_ub}) that $\mdist(\Delta) \asymp 1/M$ with high probability, while the \textit{projected minimum separation distance} is of the order $1/M^2$. Therefore, one can  recover $\mu_0$ from $M^2$ samples by either sampling the $M^2$ Fourier coefficients of lowest frequencies, or according to Theorem \ref{thm2}, by sampling $M^2$ Fourier coefficients of lowest frequencies along 3 distinct directions. 
Thus, when restricted of sampling along 3 radial lines, one attains the same sampling bounds as in the Cartesian grid sampling case. One can make a similar comparison in the second case of Theorem \ref{thm2} -- our result guarantees exact recovery  with high probability by drawing at random $\ord{M\log^2(M)}$ coefficients from $\br{-C M^2,\ldots, C M^2}$ along 3 radial lines, while the probabilistic result \cite{chi2015compressive,tang2013compressed} in the case of sampling the Fourier coefficients on a Cartesian grid requires $\ord{M\log^2(M)}$ coefficients drawn at random from the grid $\br{k\in\bbZ^2: \abs{k}\leq C' M}$.

\section{Numerical results}\label{sec:numerics}
\subsection{Verification of Theorem \ref{thm:rand_angles}}

Theorem \ref{thm:rand_angles} states that given any discrete measure $\mu_0$, one can  reconstruct  exactly by sampling frequencies in $\{-N,\ldots, N\}$ along  3 radial lines, whose angles are chosen at random from a range $S$ for which the minimum projected separation distance is sufficiently large with respect to $S$. \rev{In Figure \ref{fig:certificates}, we let 
\be{\label{eq:S}
S = \br{ (\sin(\pi t), \cos(\pi t)) : \abs{t-1/2}\leq 1/K }}
for some $K\in \bbN$ and demonstrate the reconstruction of $M$ diracs which are well separated relative to $S$. 
For their reconstruction, we sample their Fourier frequencies  along  3 directions which are chosen at random from $S$.
Letting $N = \lceil 2/\nu_{\min}\rceil$ where
$$
\nu_{\min} = \min_{\theta\in S} \min_{j\neq k} \abs{\ip{x_j - x_k}{\theta}}_{\bbT},
$$
we present the dual certificates constructed in the cases where one samples $\{-N,\ldots, N\}$ along each line, and where one chooses 50\% of the samples at random from $\{-N,\ldots, N\}$.} In all cases, the positions and amplitudes are constructed with error smaller than $10^{-5}$. \rev{Throughout this section, the error between the reconstructed positions $\hat x$ and the true positions $x$ is $\mathrm{Err}_{\mathrm{pos}} = \norm{\hat x - x}_{\ell^2}$ and the error between the reconstructed amplitudes $\hat a$  and the true amplitudes $a$ is $\mathrm{Err}_{\mathrm{amp}} = \norm{\hat a - a}_{\ell^2}/\norm{a}_{\ell^2}$.}

\begin{figure}[H]
\begin{center}
\begin{tabular}{c@{\hspace{0pt}}c@{\hspace{0pt}}c}
 \includegraphics[width = 0.3\textwidth,trim={5cm 1cm 5cm 1cm},clip]{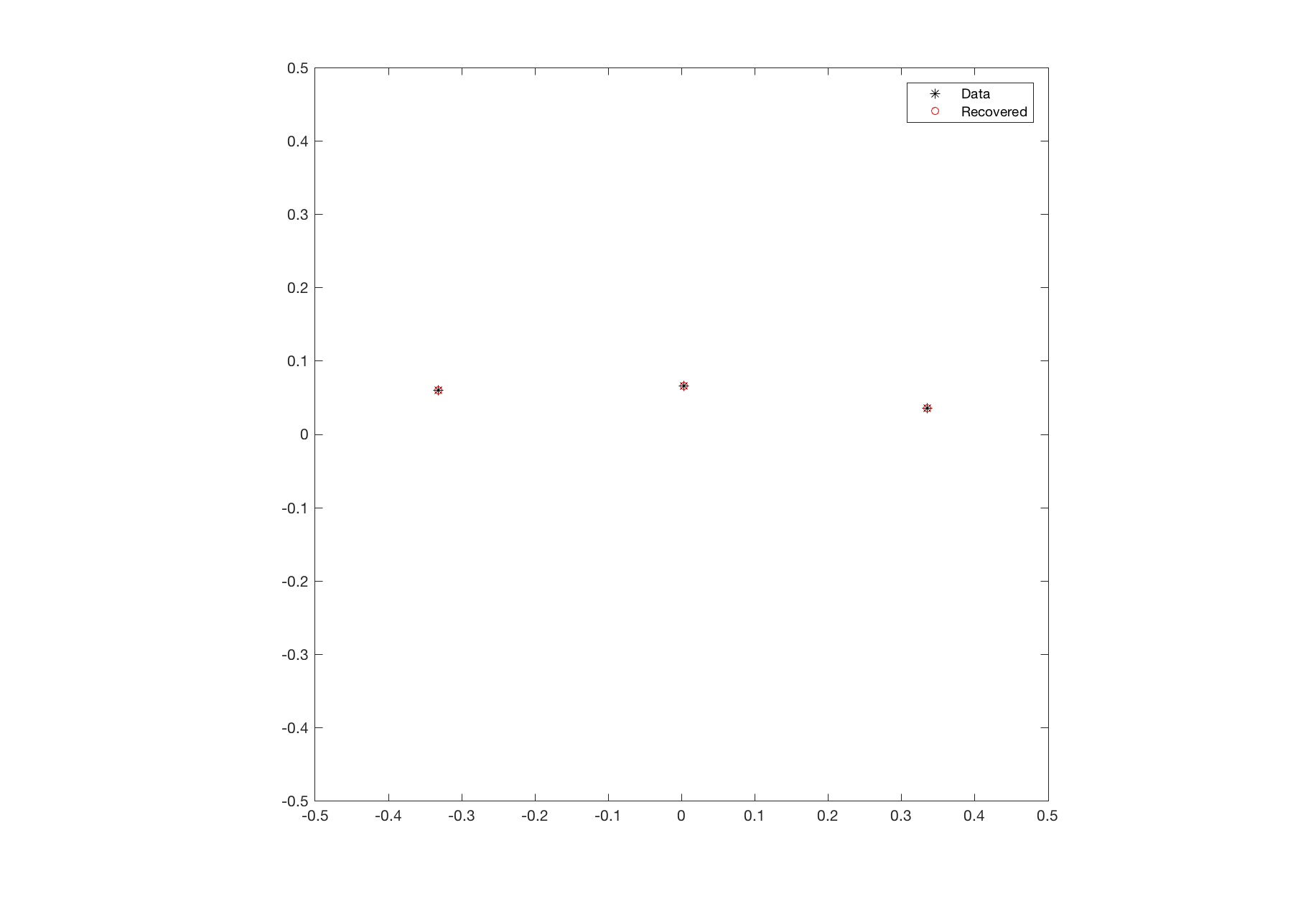}
& \includegraphics[width = 0.3\textwidth,trim={5cm 1cm 5cm 1cm},clip]{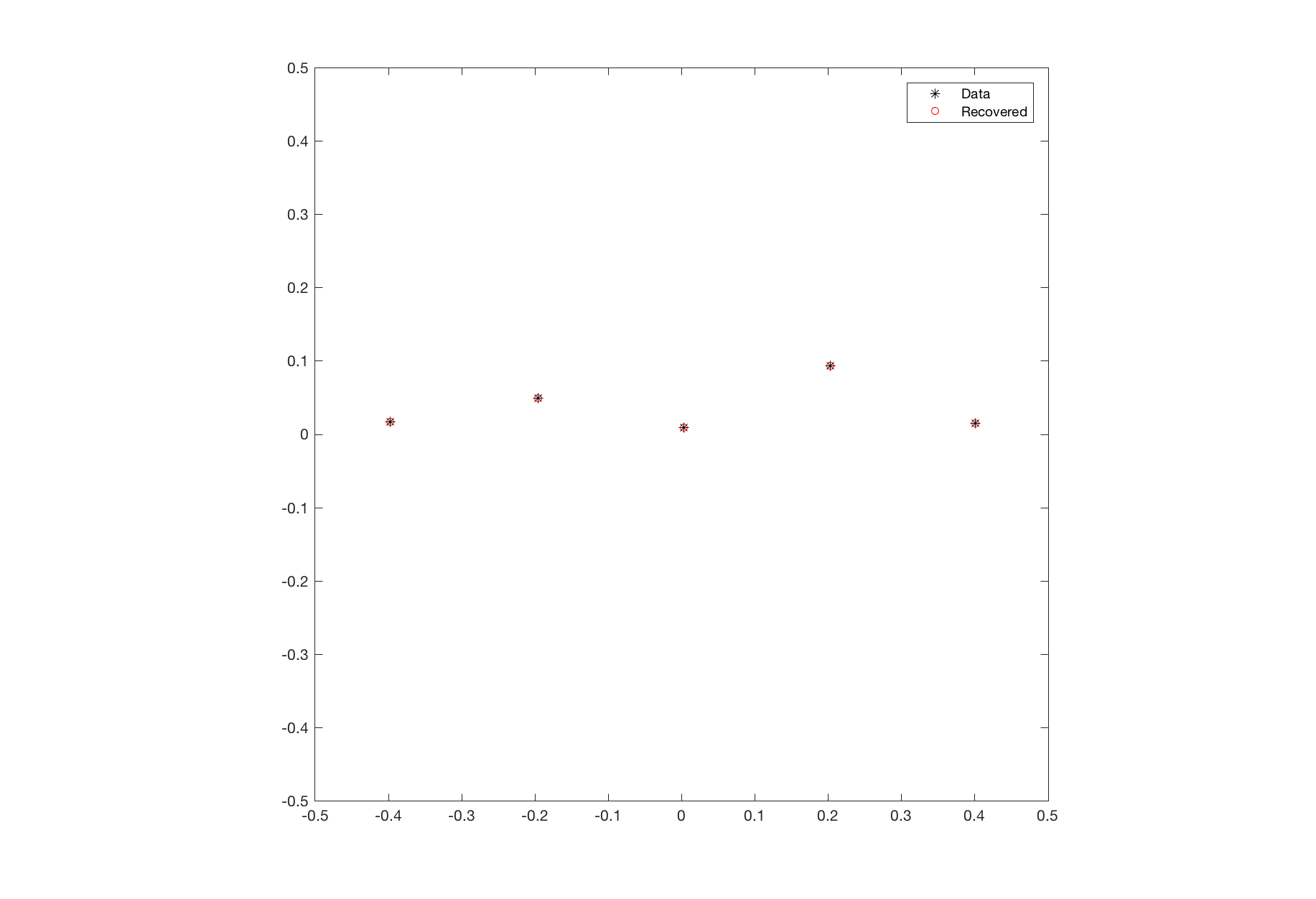}
& \includegraphics[width = 0.3\textwidth,trim={5cm 1cm 5cm 1cm},clip]{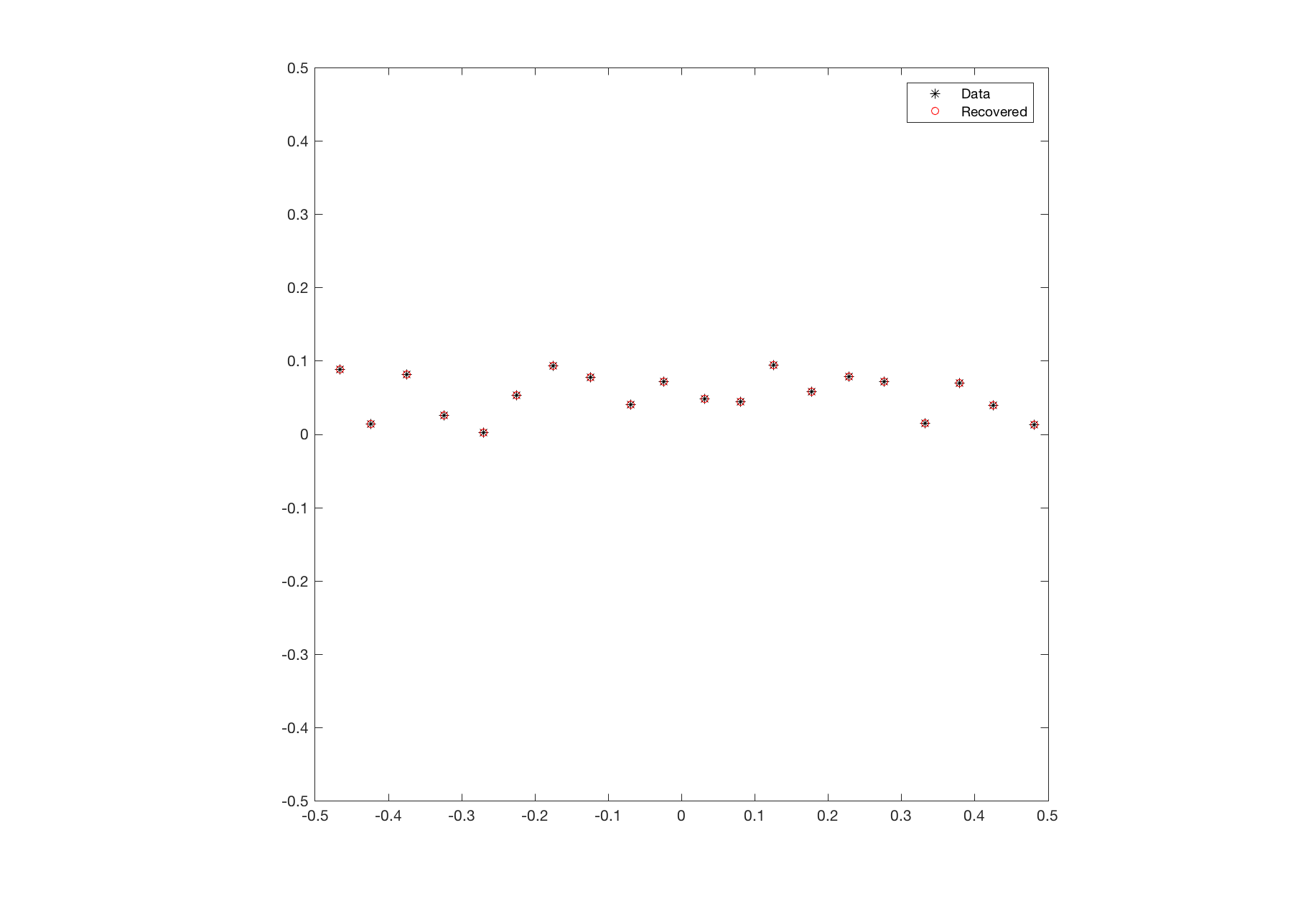} \\
 \includegraphics[width = 0.3\textwidth,trim={5cm 1cm 5cm 1cm},clip]{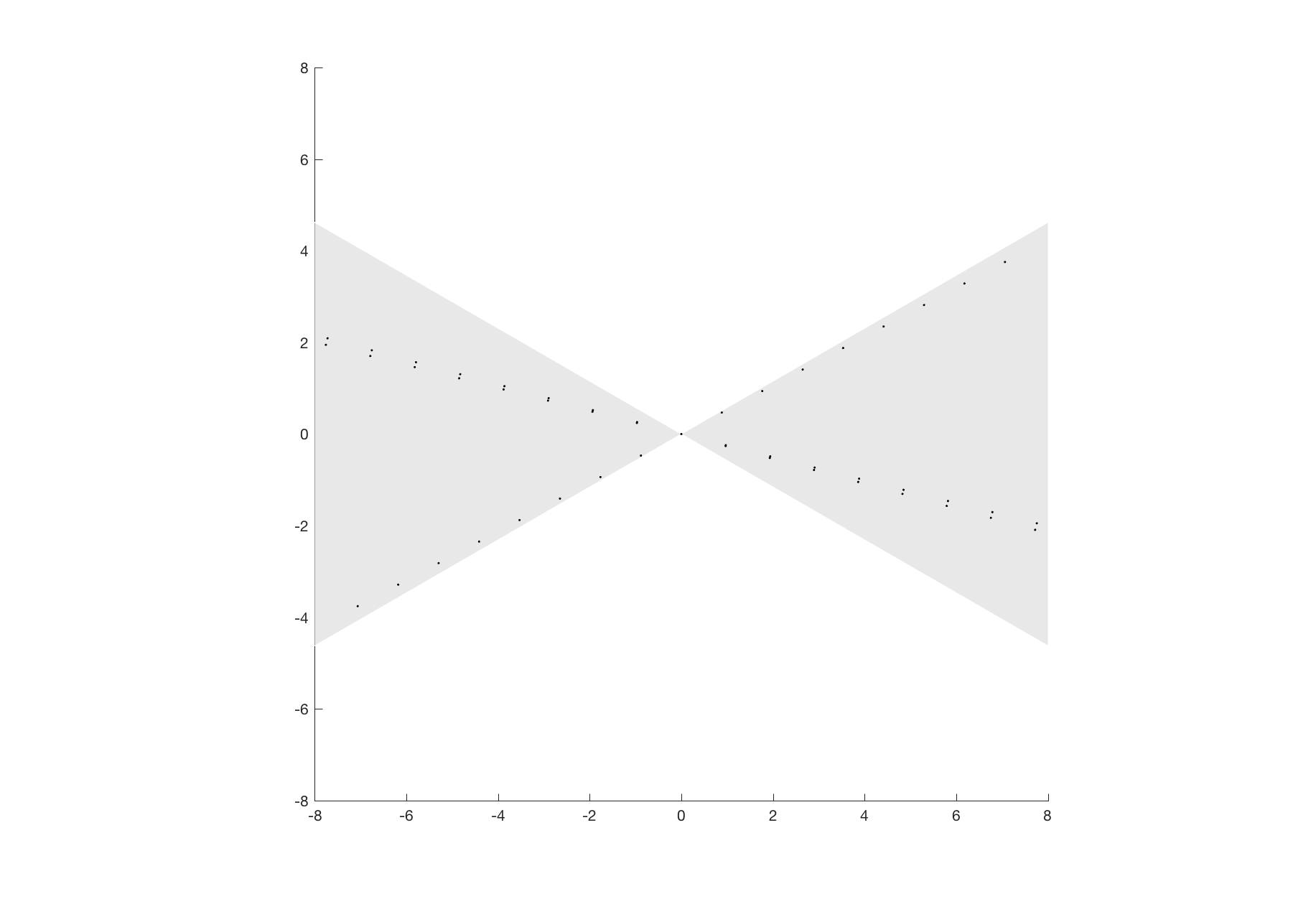}
& \includegraphics[width = 0.3\textwidth,trim={5cm 1cm 5cm 1cm},clip]{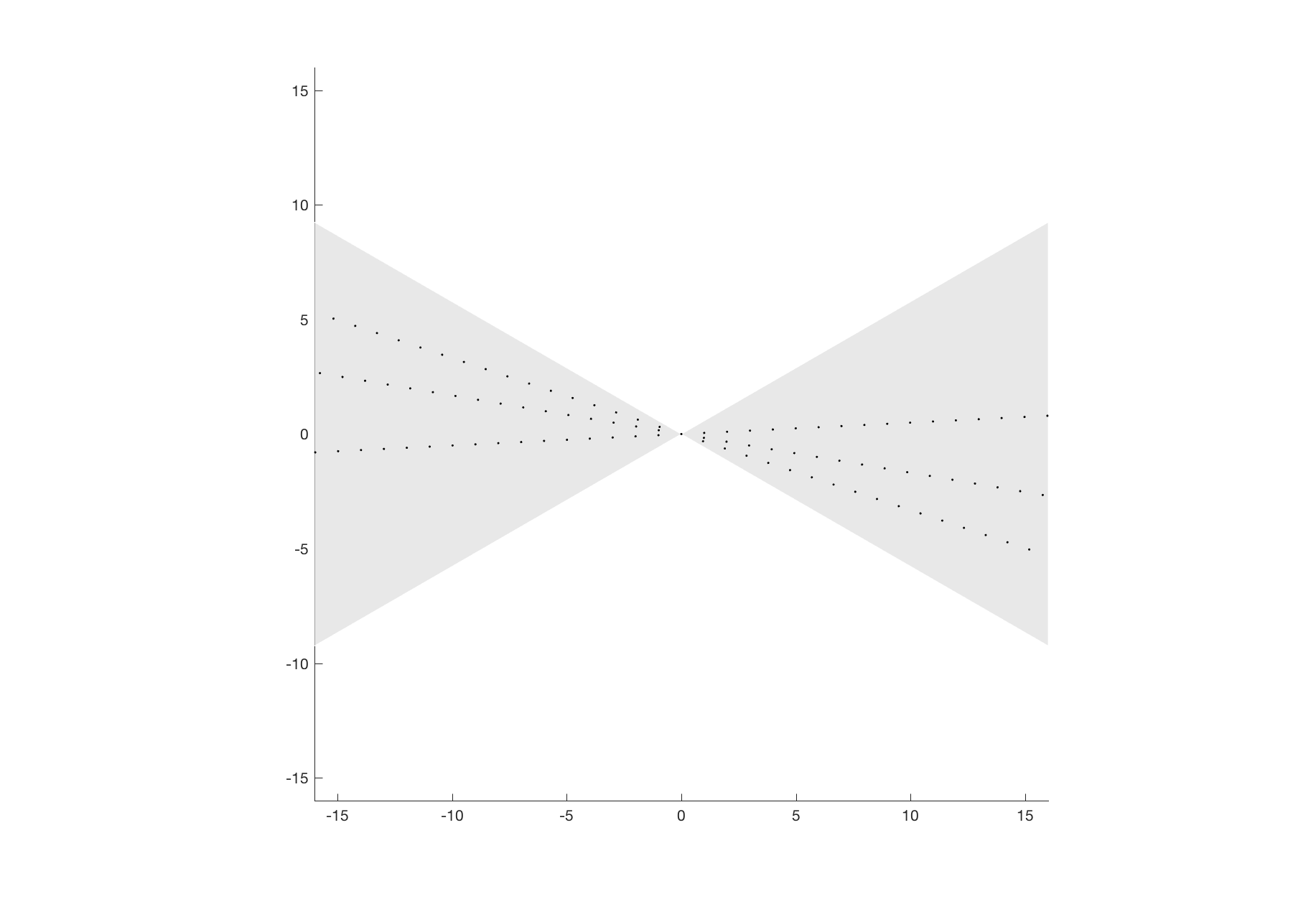}
& \includegraphics[width = 0.3\textwidth,trim={5cm 1cm 5cm 1cm},clip]{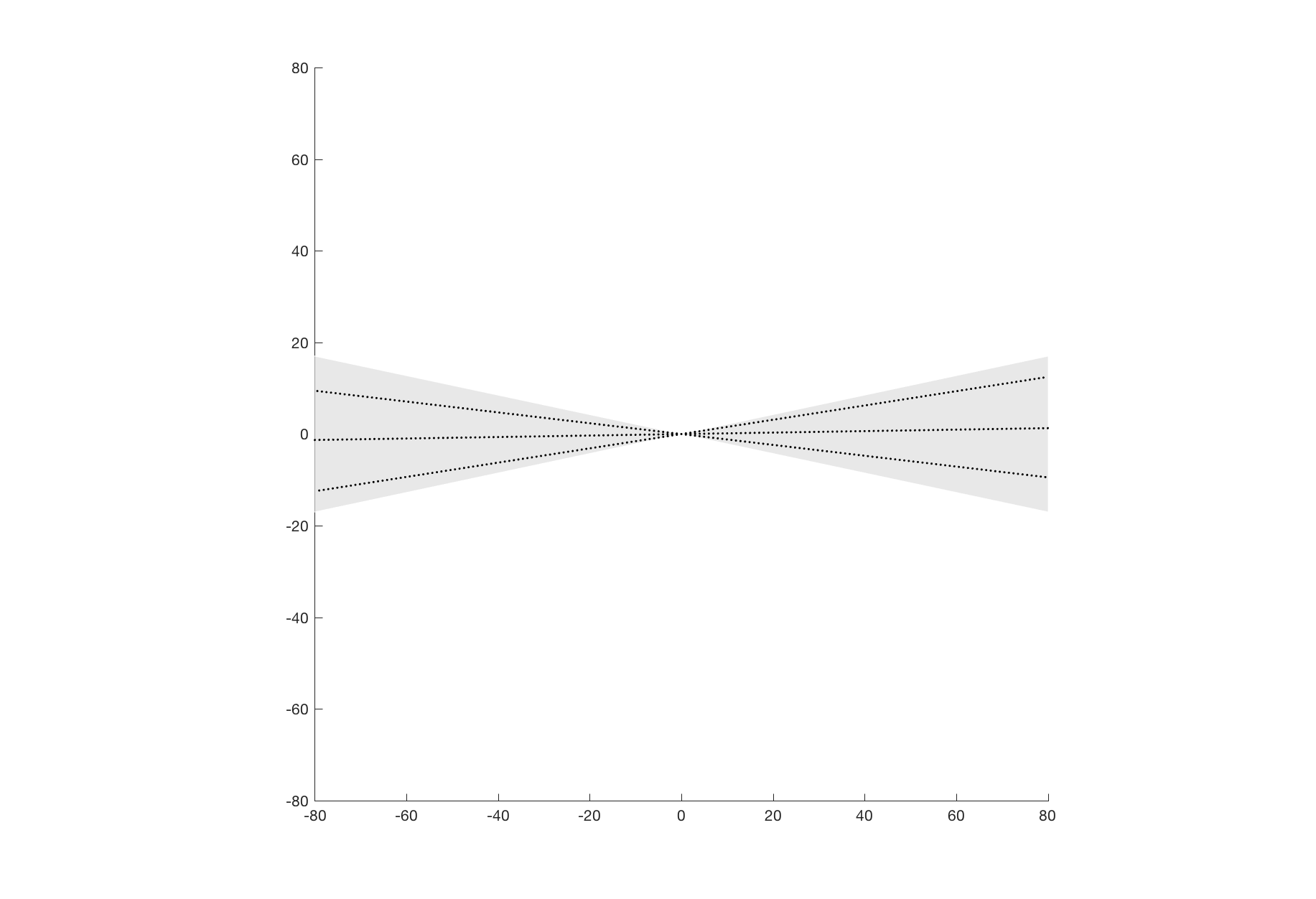} \\
 \includegraphics[width = 0.3\textwidth,trim={2cm 1cm 1cm 1cm},clip]{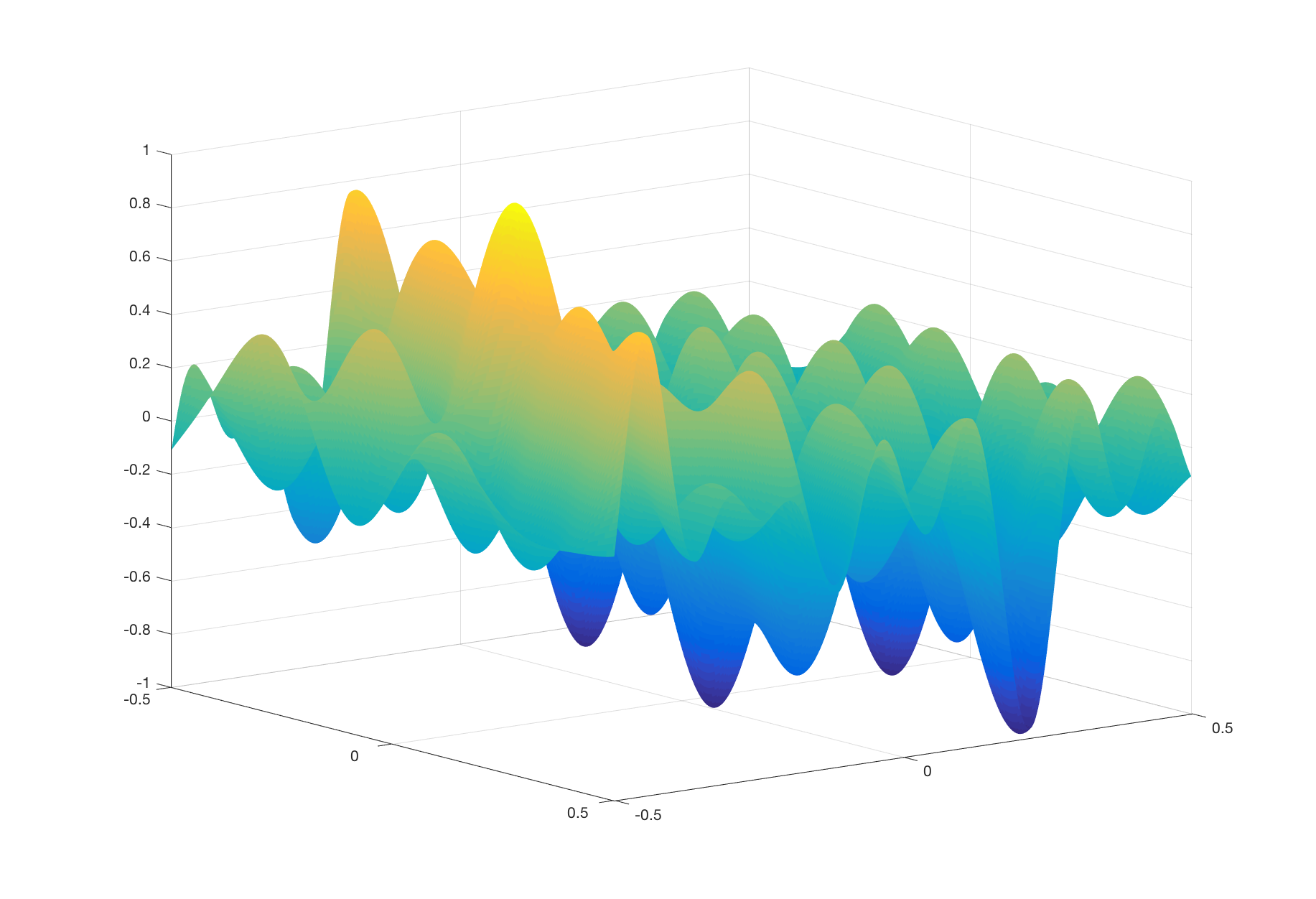}
& \includegraphics[width = 0.3\textwidth,trim={2cm 1cm 1cm 1cm},clip]{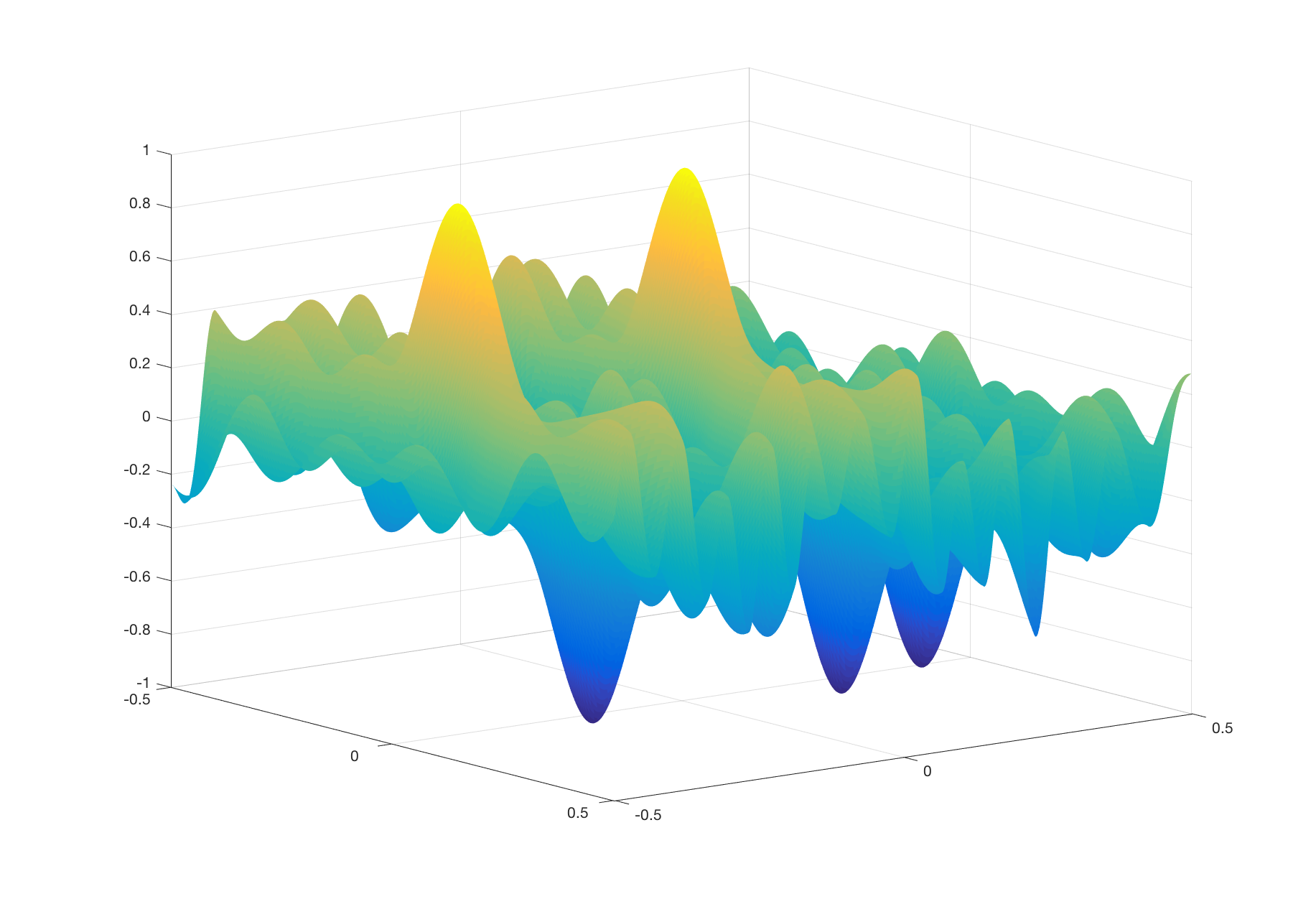}
& \includegraphics[width = 0.3\textwidth,trim={2cm 1cm 1cm 1cm},clip]{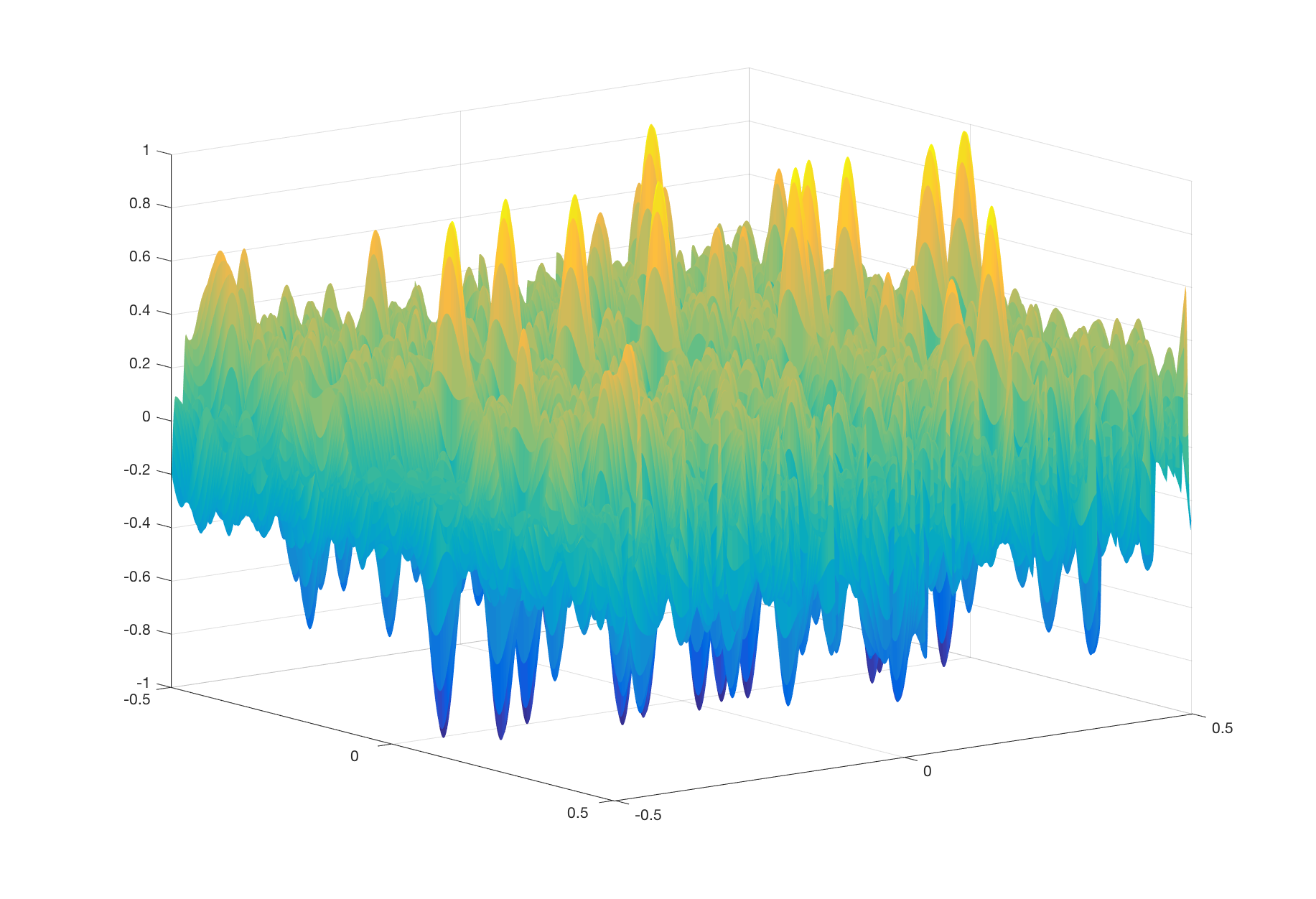}  \\
 \includegraphics[width = 0.3\textwidth,trim={2cm 1cm 1cm 1cm},clip]{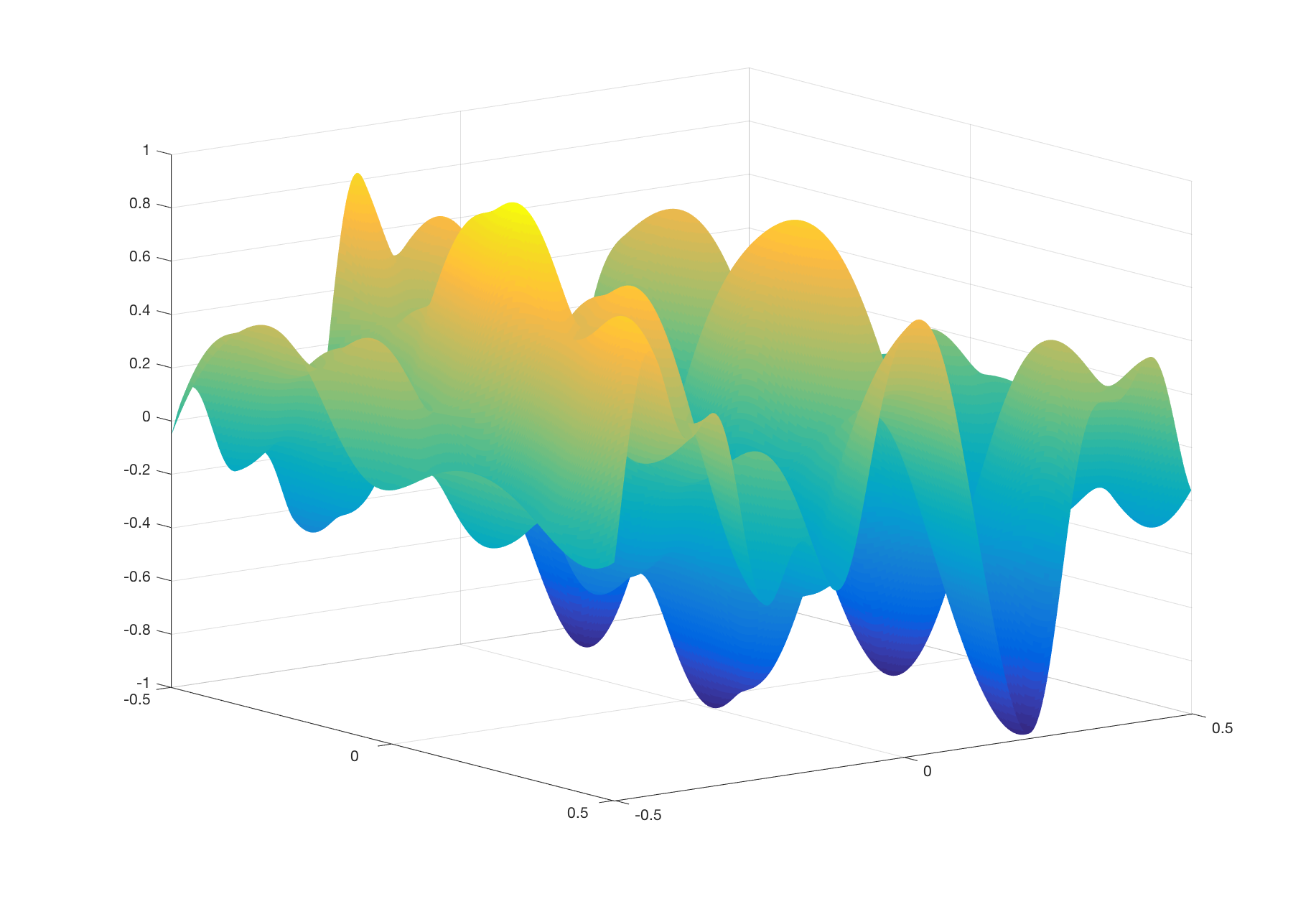}
& \includegraphics[width = 0.3\textwidth,trim={2cm 1cm 1cm 1cm},clip]{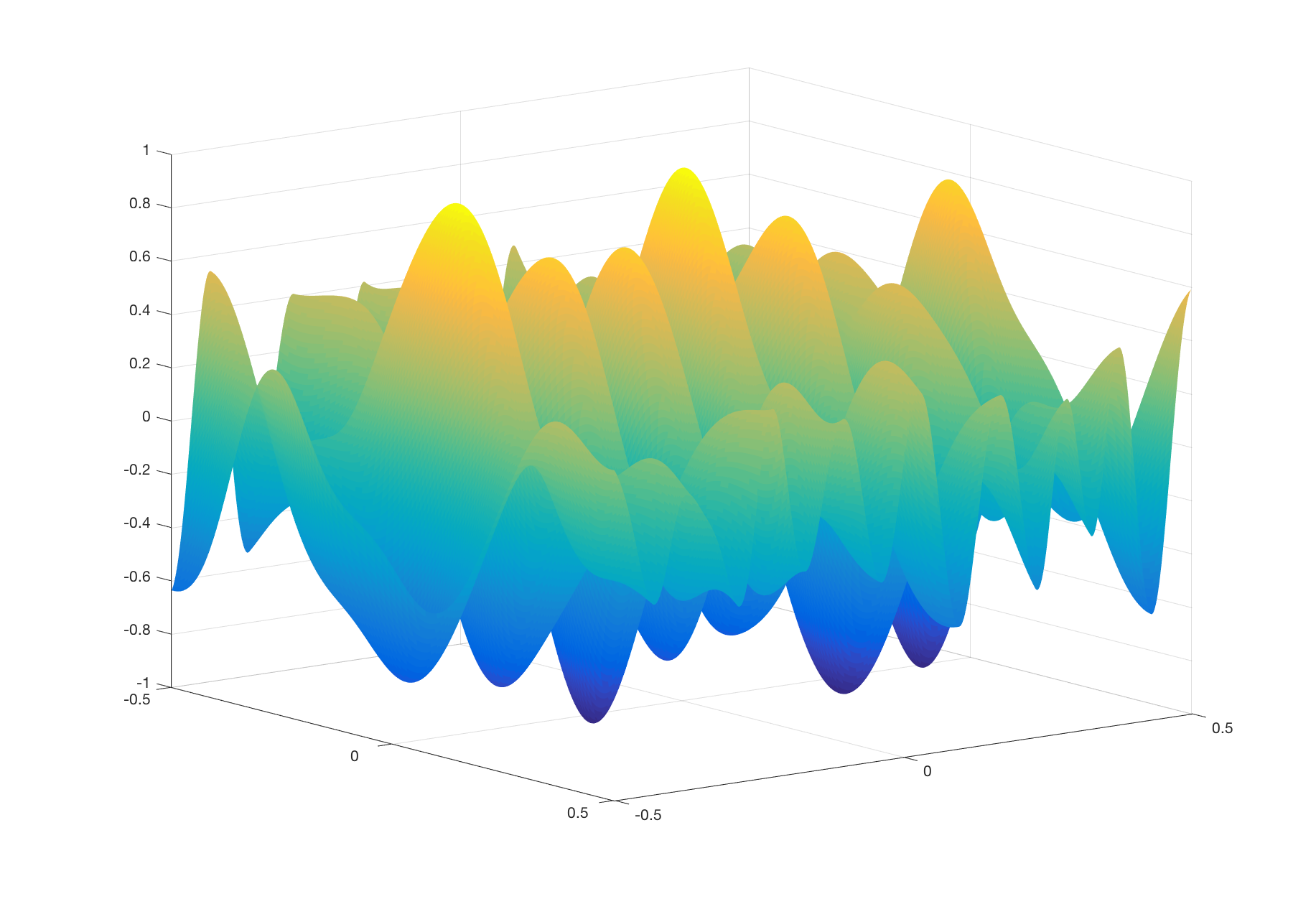}
& \includegraphics[width = 0.3\textwidth,trim={2cm 1cm 1cm 1cm},clip]{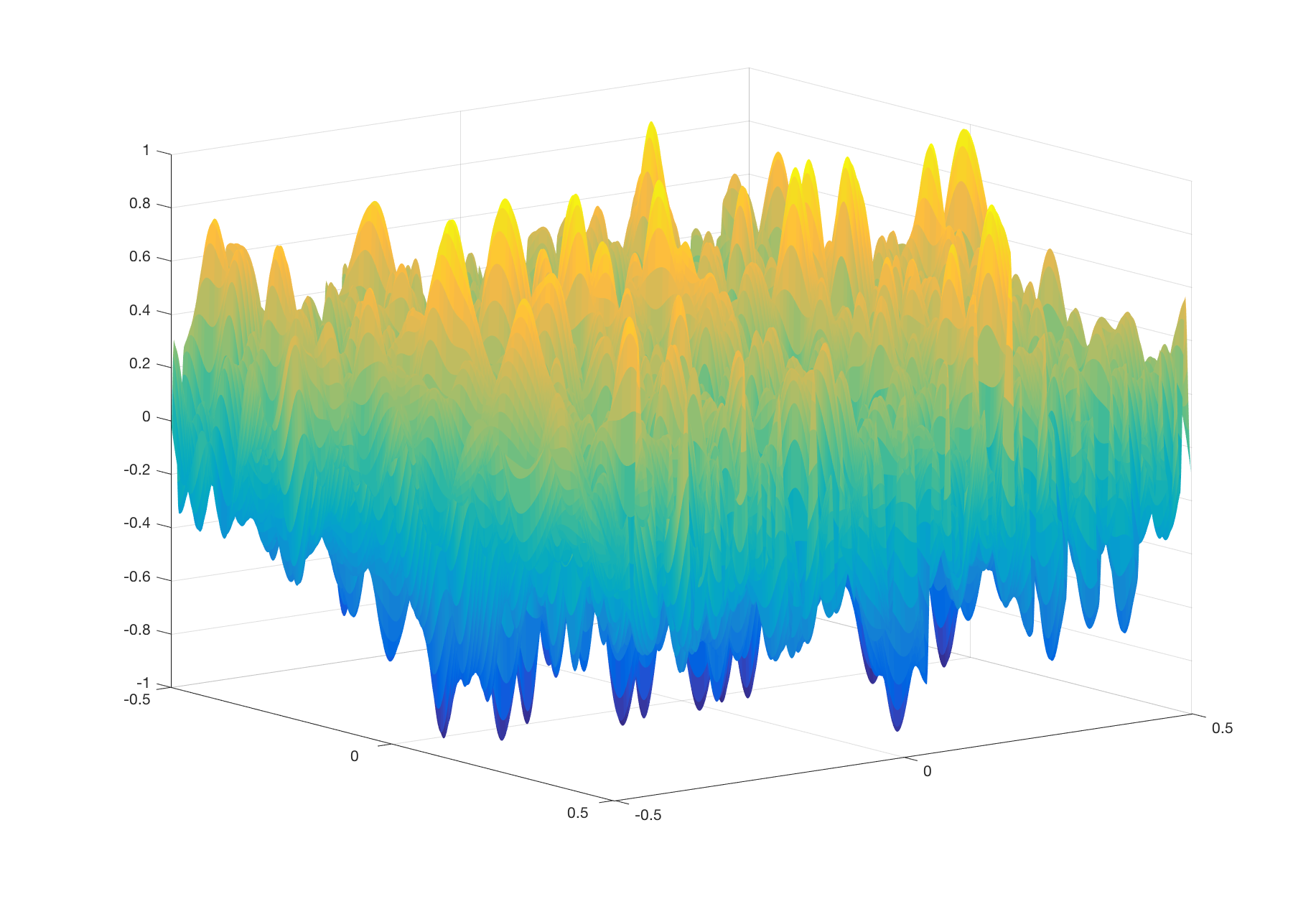}  \\
\end{tabular}
\end{center}
\caption{\rev{This figure shows the dual polynomials associated with the reconstruction of $M$ diracs using the splitting construction introduced in this article.
The far left, middle and far right columns correspond respectively to $M=3$, $M=5$ and $M=20$. In the case of $M=3$ and $M=5$, we let $K=6$ in the directions set $S$ described in \eqref{eq:S}. For this choice, $\nu_{\min } =0.2181$ and  $\nu_{\min}=0.1309$ respectively.  In the case of $M=20$, the set $S$ is chosen with $K=15$  and $\nu_{\min } = 0.0253$. The angles spanned by $S$ are indicated by the gray regions in the second row. In each case, the Fourier samples chosen along each line are from the frequencies $\{-N,\ldots, N\}$ where $N=\lceil 2/\nu_{\min}\rceil$. The indices of these Fourier samples are indicated by the black dots on the figures in the second row.  The third row and last row show respectively the dual certificates constructed when one fully samples from $\{-N,\ldots, N\}$, and  the dual certificates constructed when one samples 50\% of the frequencies in $\{-N,\ldots, N\}$, along the three directions shown in the second row.
 In all cases, the original positions and amplitudes of the diracs are reconstructed with error at most $10^{-5}$.}
\label{fig:certificates}}
\end{figure}

\subsection{Assessment of our algorithm}
\label{sec:constr_dual_cert}
As discussed in Section \ref{sec:dual_form}, the existence of a dual certificate is closely related to the existence and uniqueness of solutions to \eqref{eq:tvmin}.
This article presented one possible construction of a dual certificate by combining a sequence of univariate dual certificates. Moreover, our algorithm is guaranteed to provide the  solutions of \eqref{eq:tvmin} only when a splitting certificate exists and our main theorems provide a theoretical analysis of this. However, it is conceivable that there exists a dual certificate satisfying \eqref{toshow_unique} even when the construction of a splitting certificate is not possible.
To provide some understanding of the sharpness of our result, we shall compare numerically against the existence of another construction of a dual certificate, known as the vanishing derivatives precertificate~\cite{duval2015exact}.

One generic way of constructing a dual certificate is to consider the \textit{vanishing derivatives precertificate}, which is defined as:
\begin{equation*}
  \eta_V:= \Phi^*q_V \quad \mbox{where }q_V=\argmin \{\abs{q}^2:\ \forall i,\ (\Phi^*q)(x_{i})=\sgn(a_{i}), \diff(\Phi^*q)(x_{i})=0\},
\end{equation*}
provided the problem is feasible (which is the case if $\mu_0$ is a solution to~\eqref{eq:tvmin}). Note that $\eta_V$ is a dual certificate for $\mu_0$ if and only if $\norm{\eta_V}_\infty\leq 1$.
In this case, we say that $\eta_V$ is \emph{nondegenerate}, and it is in fact equal to the  \textit{minimal norm certificate}, that is $\eta_0:=\Phi^*q_0$ where $q_0$ is the solution to~\eqref{eq:dual_tv} with minimal $\ell^2$-norm.

The minimal norm certificate is of interest because it  governs the support stability of problem~\eqref{eq:tvmin}, that is, the properties of the support of the solutions to 
\begin{equation}\label{eq:primalnoisy}
  \inf_{\mu\in \cM(\domain)} \lambda\norm{\mu}_{TV}+ \frac{1}{2}\norm{\Phi \mu-y_0 -w}^2 \tag{$\cP_\lambda(y_0+w)$}
\end{equation}
where $w\in  \bbC^{\cfreq\times \cdir}$ is some small noise added to the observation $y_0$ and $\lambda>0$ is small.  Here, support stability refers to the property that the method recovers the same number of spikes as the original measure $\mu_0$, and that their amplitudes and locations are close to those of $\mu_0$.
We refer the reader to~\cite{duval2015exact} for more detail on the connection between the support stability and the minimal norm certificate (although
~\cite{duval2015exact} is written for a convolution, the analysis carries over without major difficulty to more general operators).

The vanishing derivatives precertificate $\eta_V$ can be seen as a proxy to understanding the stability properties of \eqref{eq:primalnoisy} and is attractive as an analytical tool since it can be computed by simply solving a linear system. In particular, 
 $q_V$ is computed by applying the Moore-Penrose pseudo-inverse of $\begin{pmatrix}
  \Phi_{x_0}&
  \Phi^{(1)}_{x_0}&
  \ldots&
  \Phi^{(d)}_{x_0}&
\end{pmatrix}^*$ to $(\sgn(a), \mathbf{0}_{dM})^T$, where  $\mathbf{0}_{dM}$ consist of $dM$ zeros,
\begin{align*}
  \Phi_{x_0}^*&=\begin{pmatrix}
     e^{2\imath\pi\ip{k_1\theta_1}{x_1}} & \ldots & e^{2\imath\pi\ip{k_{\cfreq}\theta_{\cdir}}{x_1}}\\
     \vdots & & \vdots\\
     e^{2\imath\pi\ip{k_1\theta_1}{x_{N}}} & \ldots & e^{2\imath\pi\ip{k_{\cfreq}\theta_{\cdir}}{x_N}}
   \end{pmatrix},\\
   \Phi_{x_0}^{(\ell)*}&=2\imath\pi\begin{pmatrix}
 \ip{k_1\theta_1}{v_\ell} e^{2\imath\pi\ip{k_1\theta_1}{x_1}} & \ldots & \ip{k_{\cfreq}\theta_{\abs{\setdir}}}{v_\ell}e^{2\imath\pi\ip{k_{\cfreq}\theta_{\cdir}}{x_1}}\\
     \vdots & & \vdots\\
   \ip{k_1\theta_1}{v_\ell}  e^{2\imath\pi\ip{k_1\theta_1}{x_{N}}} & \ldots &  \ip{k_{\cfreq}\theta_{\cdir}}{v_\ell}e^{2\imath\pi\ip{k_{\cfreq}\theta_{\cdir}}{x_N}}
   \end{pmatrix},
\end{align*}
and $(v_1, \ldots, v_{d})$ is an orthonormal basis of $\bbR^d$. Note that the minimal norm certificate and vanishing derivatives precertificate can be constructed only when  the underlying measure $\mu_0$ is known, so they should be seen as analytic tools to understanding when one can recover $\mu_0$ in a stable manner.

Numerically, in the case of sampling the Fourier transform on a grid, $\eta_V$ has been numerically observed to be nondegenerate whenever the positions of the point sources are $\ord{1/f_c}$ apart, where $f_c$ is the range of the sampled Fourier coefficients.  This suggests that $\eta_V$ yields an accurate understanding of stability and recoverability. \rev{Furthermore, we remark that theoretical results on conditions under which $\eta_V$ is nondegnerate have been derived in the case of weighted Fourier samples \cite{li2016approximate} and convolutional sampling operators \cite{bendory2014robust}}. Therefore, in what follows, we shall compute the corresponding vanishing derivatives precertificate and regard its nondegeneracy as a means of checking whether there are cases of recoverability where  our splitting certificate cannot be constructed.

\paragraph{Example}

In Figure \ref{t:skew}, we present the experimental results where one samples along the directions
\rev{$$
\{(\sin(\pi t), \cos(\pi t) ) : t\in \{1/2-1/7, 1/2, 1/2+1/7\} \}.
$$}
Since the directions are closer to the vertical axis, as suggested by Theorem \ref{thm2}, we shall consider the recovery of point sources whose positions are drawn from a distribution which favours concentration along the horizontal axis.
\revision{For each $M$, we generate 200 signals at random,} whose amplitudes are drawn at random from the range $[-55,55]$ and whose positions are
$$
\br{ \alpha_j(\cos(\beta_j), \sin(\beta_j)) : j=1,\ldots, M}, 
$$
where $\beta_j \sim \cN(0,0.005)$ is drawn from a normal distribution with mean 0 and variance 0.005, and $\alpha_j$ is drawn from the uniform distribution on $[-1/2, 1/2]$. Along each of the radial lines, the Fourier samples drawn are either those indexed by $\{-N,\ldots, N\}$ or $30\%$ of those indexed by $\{-N,\ldots, N\}$. The graphs show the fraction of these 200 signals which are exactly reconstructed (by exact, we mean that the reconstruction error of the positions and amplitudes are at most $10^{-4}$). For comparison, we also compute the vanishing derivatives precertificate in each case and record the percentage of the signals for which the vanishing derivatives precertificate is nondegenerate. As mentioned in Section \ref{sec:constr_dual_cert}, nondegeneracy of the vanishing derivatives precertificate is generally an indication that one can stably recover a measure via TV minimization. Furthermore, note that the construction of the vanishing derivatives precertificate requires knowledge of the original measure, and thus can be seen as the `ground truth' on which signals can be recovered. The similarity between  the fraction of successful recovery via our splitting certificate and the fraction of signals for which  the vanishing derivatives precertificate is nondegenerate suggests that the reconstruction via the SDP reconstruction  procedure is close to that of the true TV minimization problem when the underlying measure can be reconstructed via TV minimization.

\begin{figure}[H]
\begin{center}
\begin{tabular}{c@{\hspace{0pt}}c@{\hspace{0pt}}c@{\hspace{0pt}}c}
$M=3$ &$M=4$\\
 \includegraphics[width = 0.4\textwidth,trim={2cm 1cm 1cm 1cm},clip]{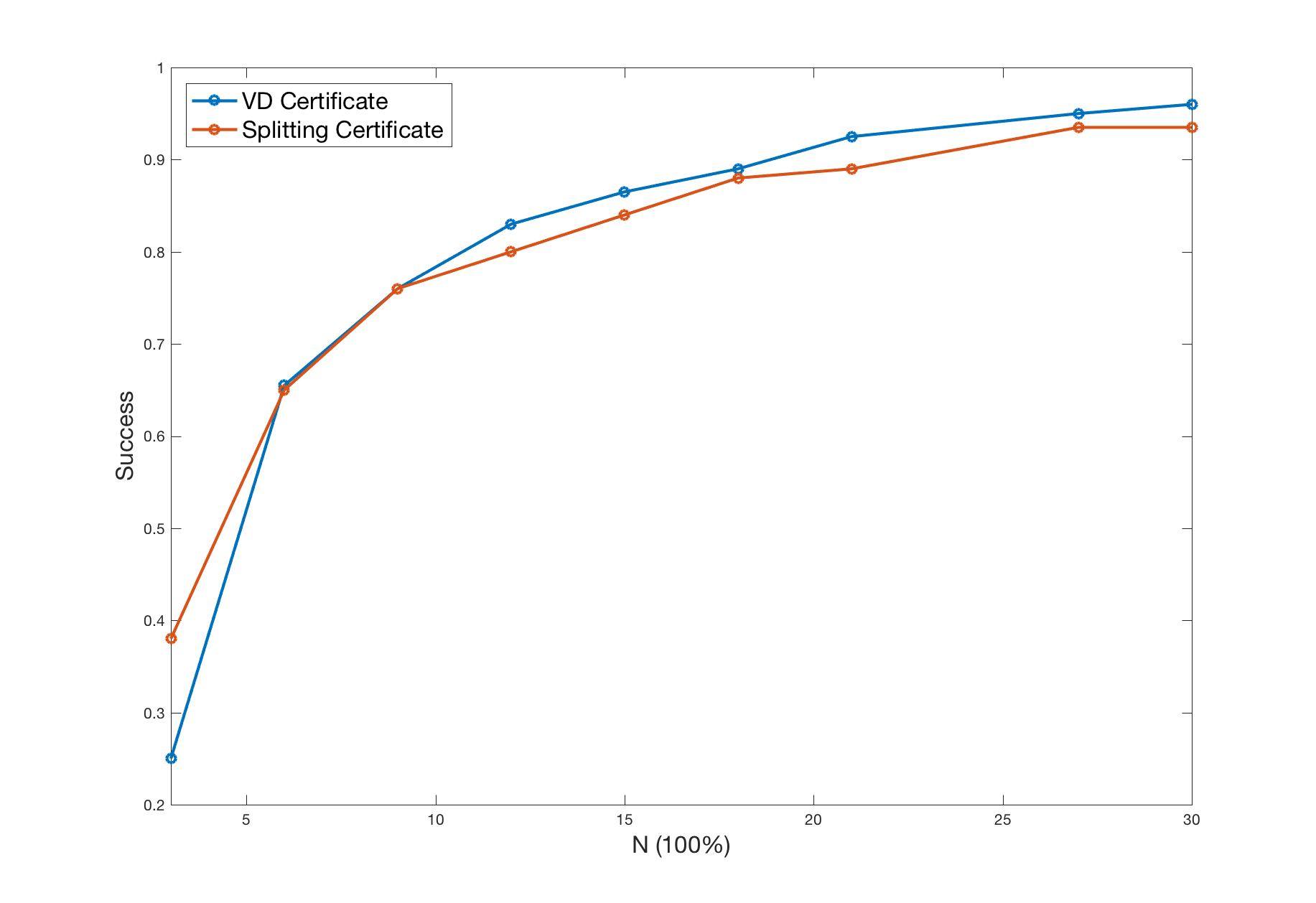}
 &
 \includegraphics[width = 0.4\textwidth,trim={2cm 1cm 1cm 1cm},clip]{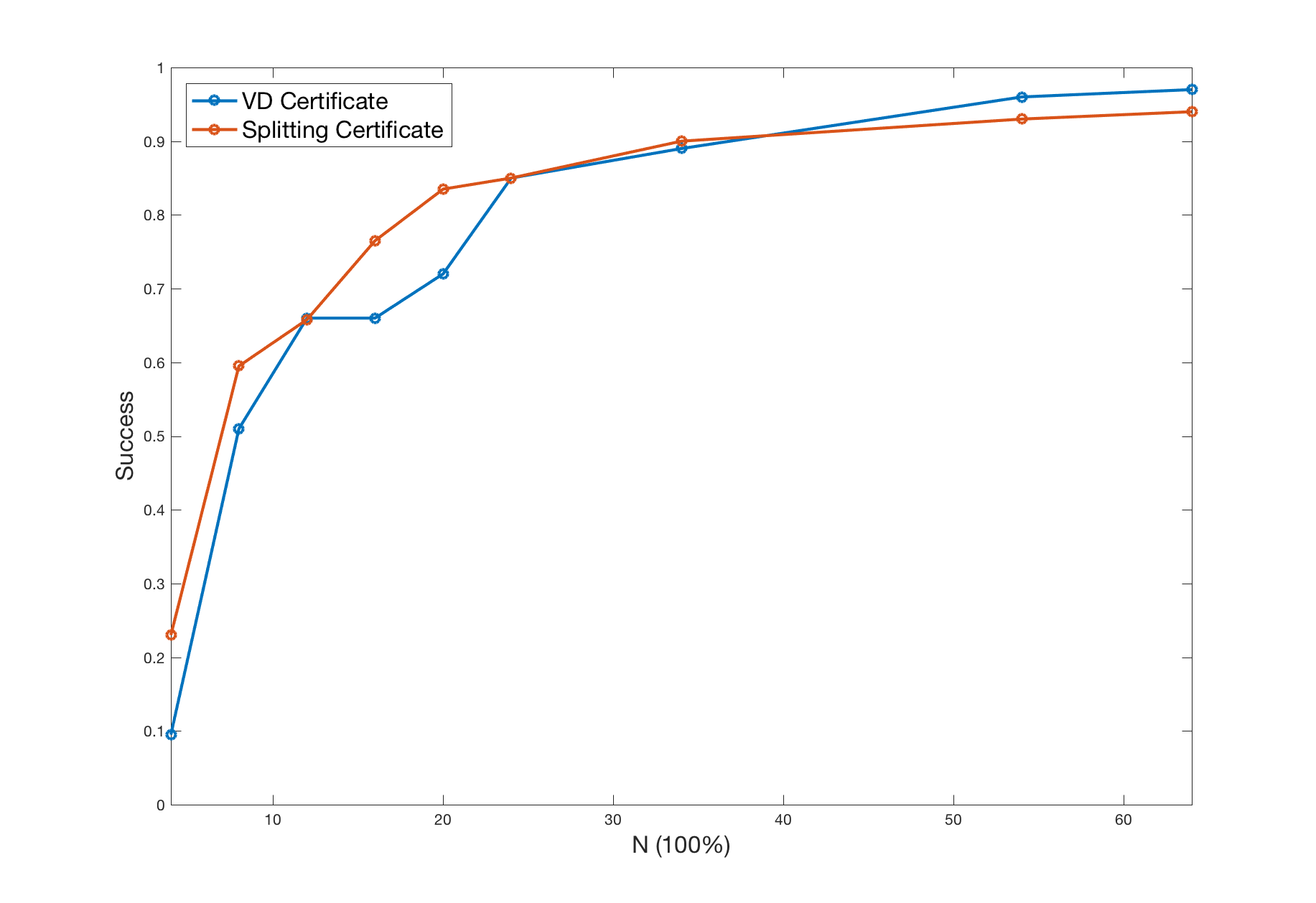}\\
 \includegraphics[width = 0.4\textwidth,trim={2cm 1cm 1cm 1cm},clip]{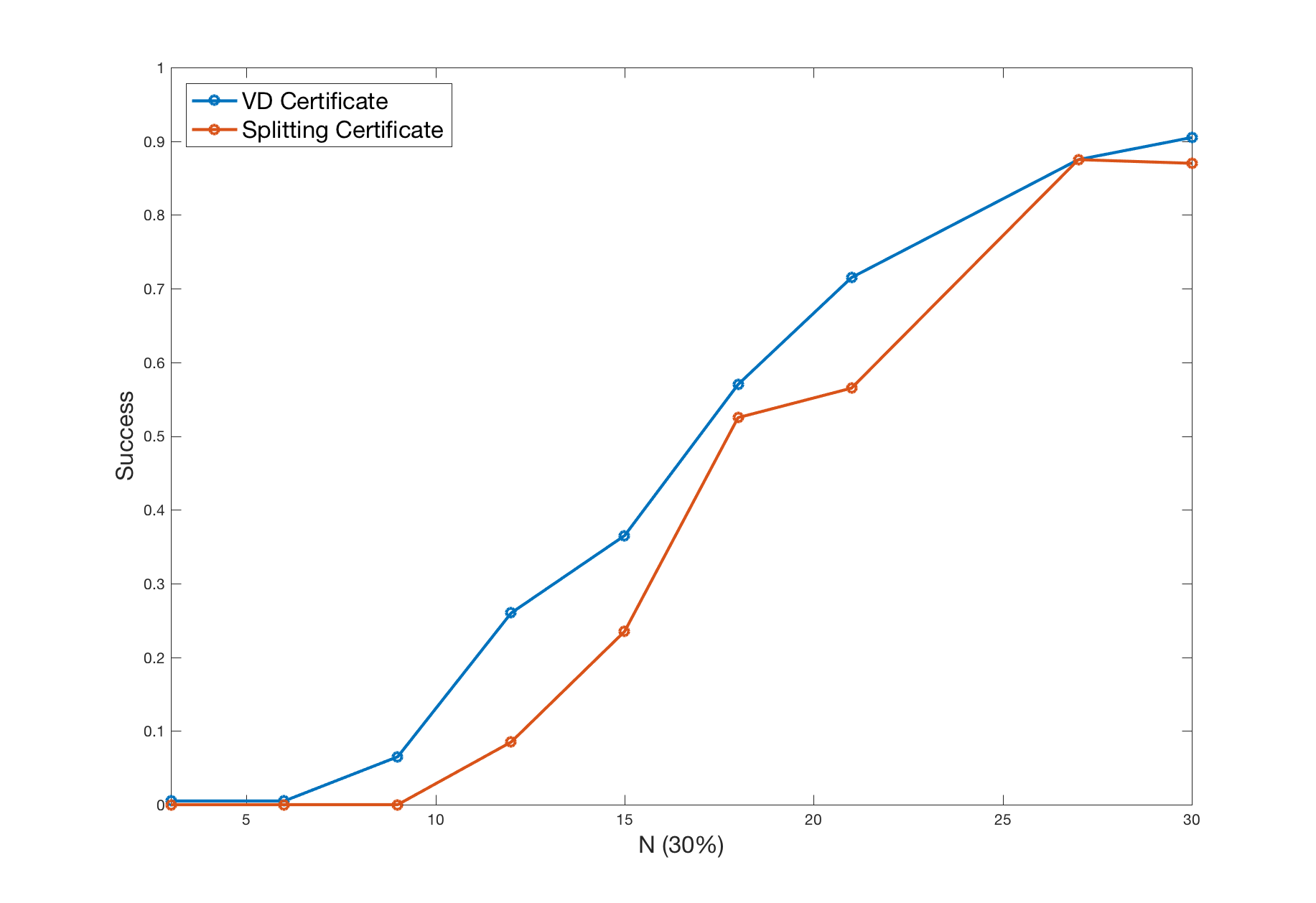}
& \includegraphics[width = 0.4\textwidth,trim={2cm 1cm 1cm 1cm},clip]{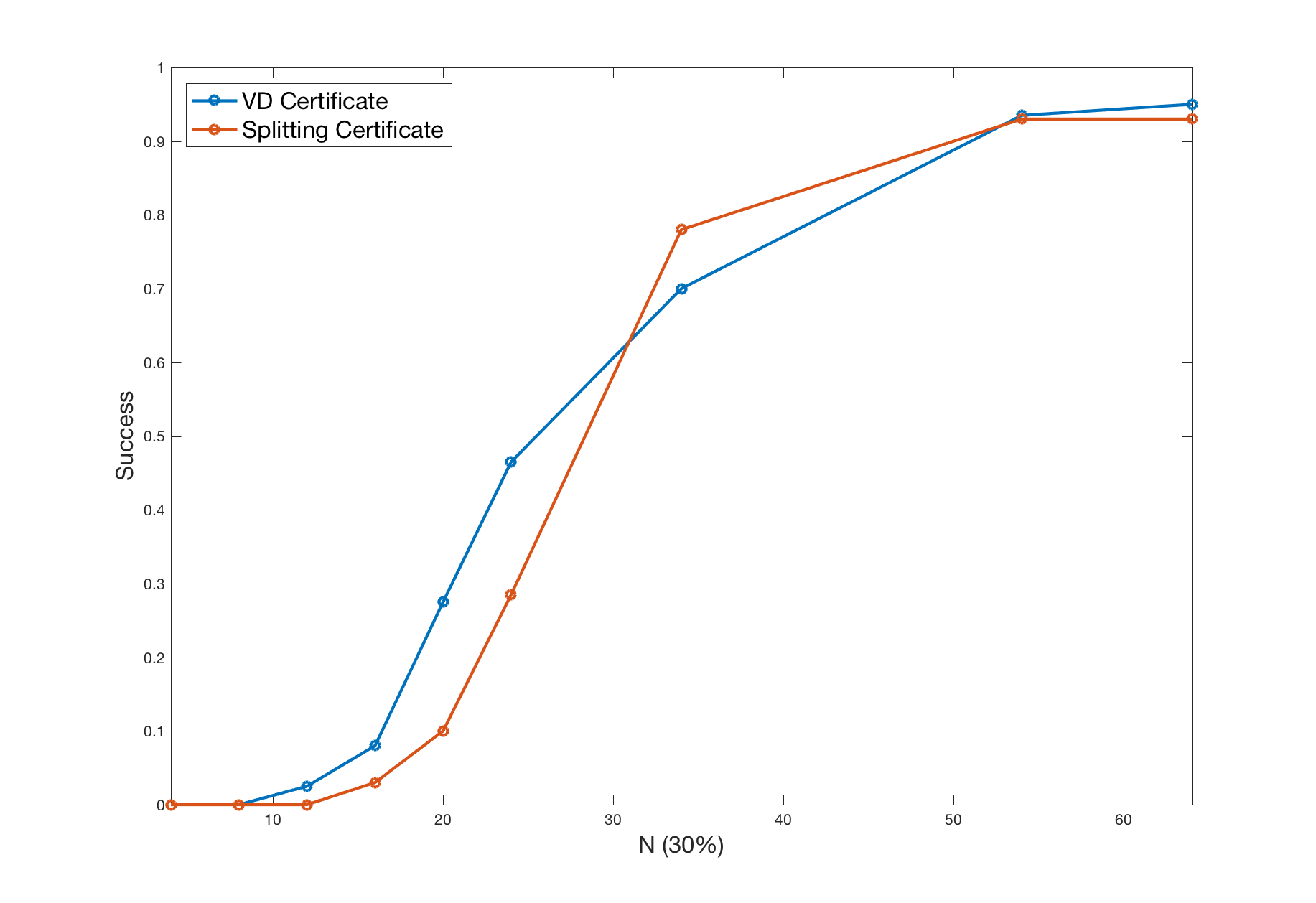} \\
\end{tabular}
\end{center}
\caption{Fraction of the 200 signals which are successfully recovered via the splitting dual approach, and the fraction for which the vanishing derivatives precertificate is nondegenerate. The top row corresponds to sampling $\{-N,\ldots, N\}$ along 3 directions and the bottom row corresponds to subsampling 30\% of $\{-N,\ldots, N\}$ along 3 directions \label{t:skew}}
\end{figure}

\subsection{Verification of Theorem \ref{thm2}}\label{sec:comparison}

In this experiment, we observe Fourier samples drawn along 3 fixed radial lines, directed by
$$
\{(\sin(\pi t), \cos(\pi t)) : t\in \{0,1/3, 2/3\} \}.
$$ 
For each $M$, we consider the reconstruction of 200 signals  whose positions are randomly chosen in the domain $B(0,1/2)$ and whose amplitudes are randomly chosen in $[-55,55]$, in both instances, with the uniform distribution. 
Figure \ref{t:exp2_100} shows the reconstruction results for the case where the
 Fourier samples observed along each line are indexed by $\{-N,\ldots, N\}$. Figure \ref{t:exp2_30} shows the reconstruction results for the case where we observe along each line $30\%$ of the Fourier samples indexed by $\{-N,\ldots, N\}$.
For comparison, we show the fraction of signals for which the corresponding vanishing derivatives precertificate is nondegenerate.

As suggested by Theorem \ref{thm2}, one is required to sample at least $M^2$ Fourier frequencies along each line to guarantee recovery with high probability. Notice that the fraction of successful recovery via the spitting approach with $L'=3$ is lower than the fraction of signals for which the vanishing derivatives precertificate is nondegenerate, whereas, choosing $L'=2$ results in reconstruction rates which are roughly in line with the nondegeneracy of the vanishing derivatives precertificate. To see why choosing $L'=2$ yields better results, first note that it is easier for condition (A1) to be satisfied when $L'$ is smaller -- since choosing $L'=2$ means that we allow for the nonexistence of a   dual certificate  along one of the 3 chosen directions, one can choose a slightly lower sampling range. Furthermore, in the case where we have only two angles for which (A1) holds, the set $\supsat$ for which (A2) is required to hold consists of up to $M^2$ points. However, since each line consists of  $\ord{M^2}$ sampling points, one can in fact show that the operator $A_{\Theta,\Gamma,\supsat}$ is injective almost surely when the positions are drawn at random.
The closeness of the recovery rates  for this splitting certificate with the `ground truth' recovery rates as approximated by the vanishing derivatives precertificate suggest that further analysis of this splitting certificate may lead to a deeper understanding of the solutions of \eqref{eq:tvmin}.

\begin{figure}[H]
\begin{center}
\begin{tabular}{c@{\hspace{0pt}}c@{\hspace{0pt}}c@{\hspace{0pt}}c}
$M=3$ &$M=4$\\
 \includegraphics[width = 0.4\textwidth,trim={2cm 1cm 1cm 1cm},clip]{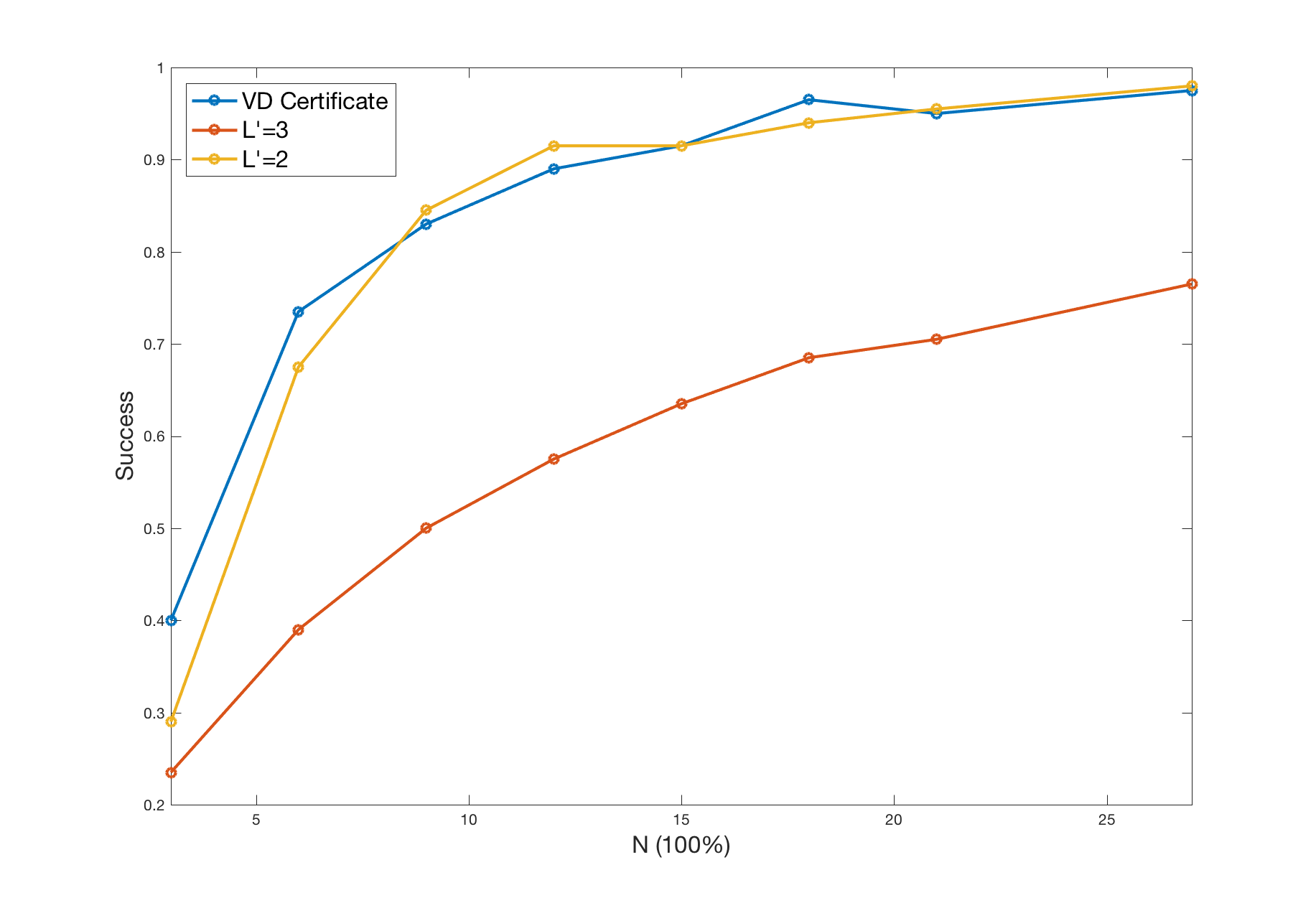}
 &
 \includegraphics[width = 0.4\textwidth,trim={2cm 1cm 1cm 1cm},clip]{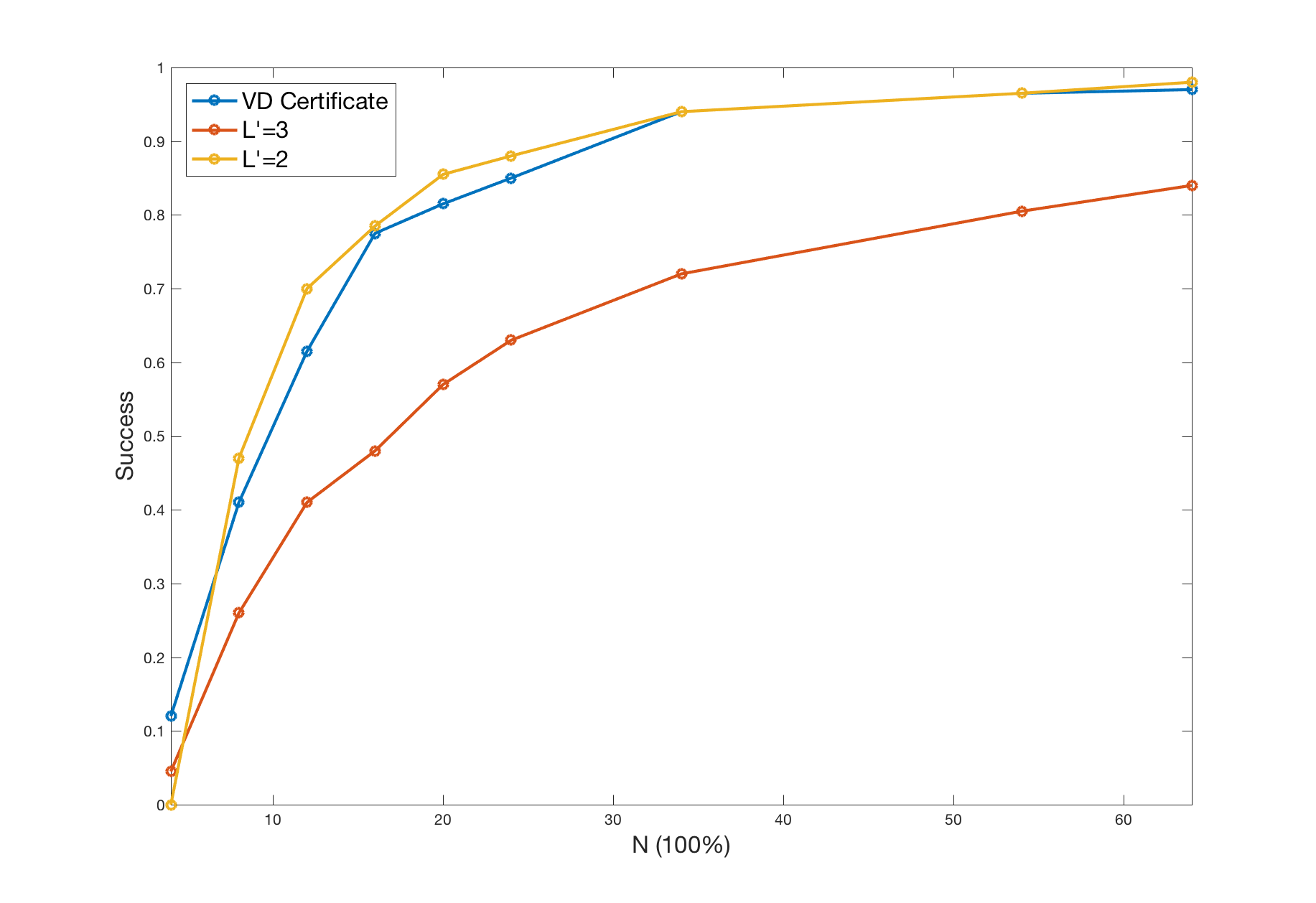}
\end{tabular}
\end{center}
\caption{Fraction of the 200 signals which are successfully recovered via the splitting dual approach, and the fraction for which the vanishing derivatives precertificate is nondegenerate. Here, we sample all  Fourier coefficients indexed by $\{-N,\ldots,N\}$ along angles $\{0,\pi/3, 2\pi/3\}$. \label{t:exp2_100}}
\end{figure}

\begin{figure}[H]
\begin{center}
\begin{tabular}{c@{\hspace{0pt}}c@{\hspace{0pt}}c@{\hspace{0pt}}c}
$M=3$ &$M=4$\\
 \includegraphics[width = 0.4\textwidth,trim={2cm 1cm 1cm 1cm},clip]{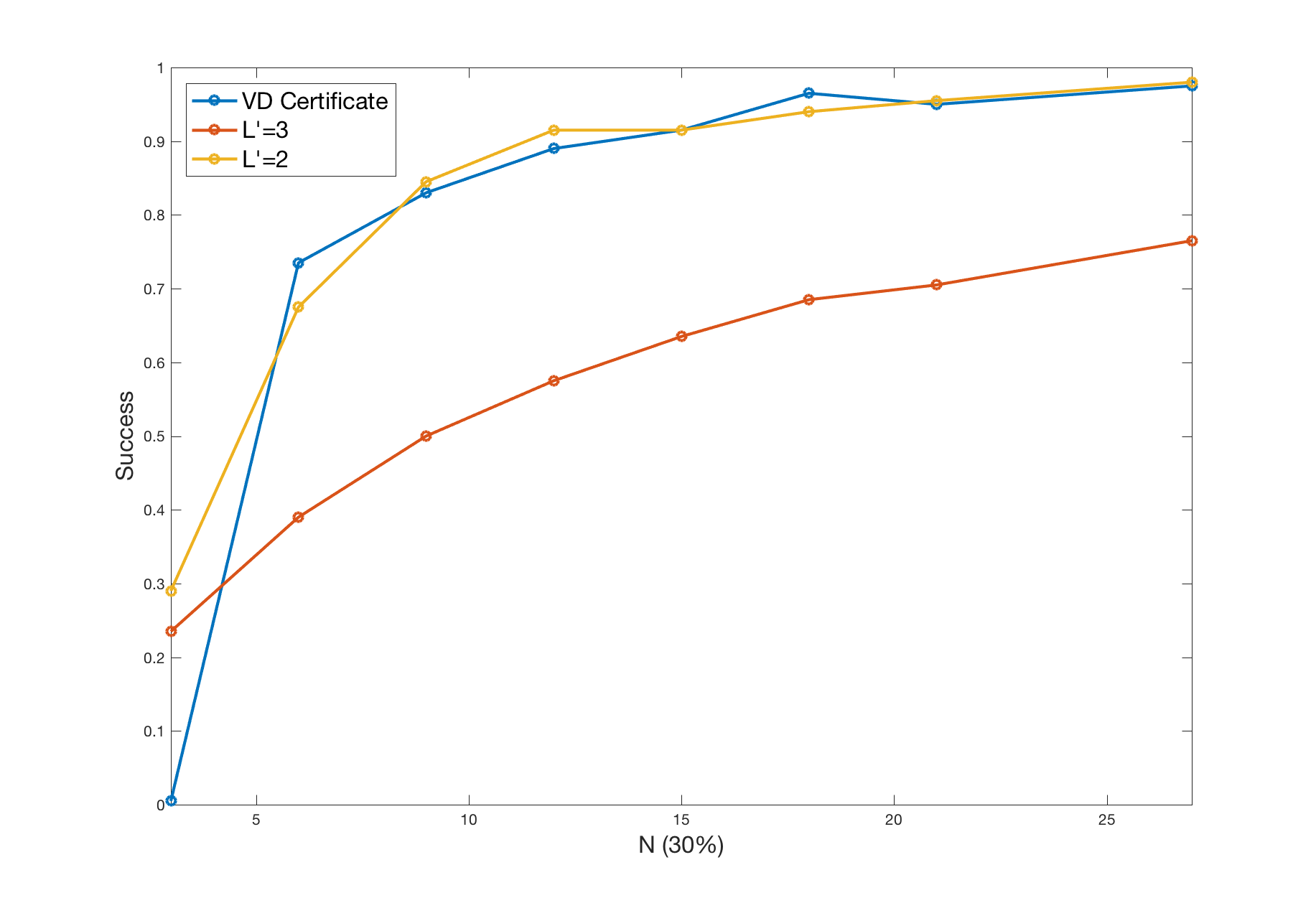}
 &
 \includegraphics[width = 0.4\textwidth,trim={2cm 1cm 1cm 1cm},clip]{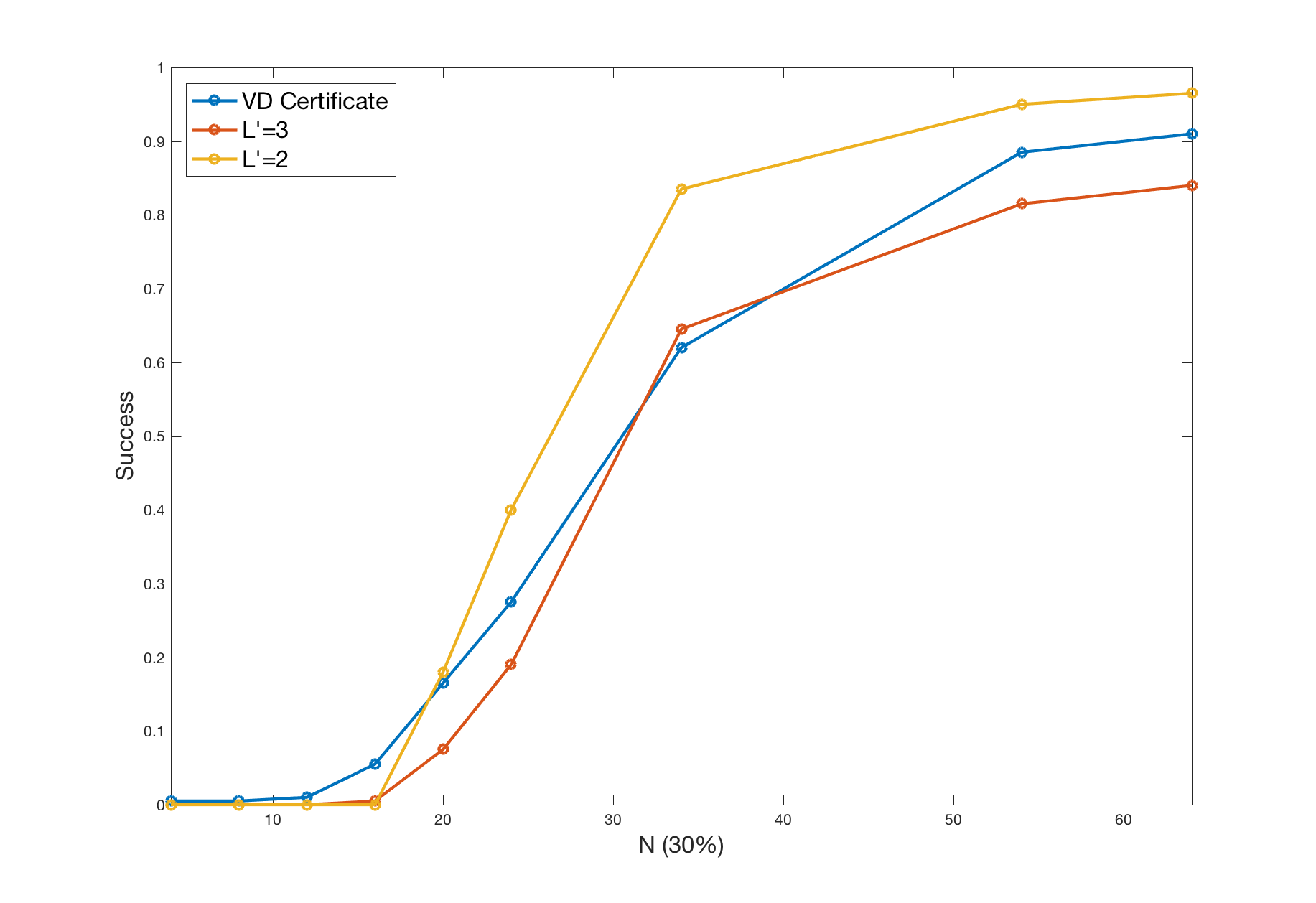}
\end{tabular}
\end{center}
\caption{Fraction of the 200 signals which are successfully recovered via the splitting dual approach, and the fraction for which the vanishing derivatives precertificate is nondegenerate. Here, we sample 30\% of the Fourier coefficients in $\{-N,\ldots,N\}$ along angles $\{0,\pi/3, 2\pi/3\}$. \label{t:exp2_30}}
\end{figure}

\section{Extension to the noisy setting}\label{sec:noise}
\rev{
In this section, we  explain how algorithm \algname~ can be extended for the recovery of $\mu_0 = \sum_{j=1}^M a_j \delta_{x_j}$, given noisy measurements $y = y_0  + w$ where $y_0 = \Phi \mu_0$ and $w\in \bbC^{T\times L}$ is such that $\norm{w}_2\leq \delta$ for some small $\delta>0$.
However, let us stress the fact that the proposed method does not coincide with the solution of~\eqref{eq:primalnoisy} for $\lambda>0$.
}

\rev{
First note that the dual problem of  the robust TV minimization problem  \eqref{eq:primalnoisy}
is
$$
\sup\br{\Re\ip{y_0}{q} - \frac{\lambda}{2}\norm{q}_2^2:\; \norm{\Phi^* q}_\infty\leq 1}.
$$
This naturally suggests that we can extend Algorithm \algname{} to the noisy setting by carrying out the following modifications:
\begin{enumerate}
\item[(I)] Replace each SDP problem in Step 1 by its robust version:
\be{\label{eq:blassodual1D}
\sup\left\{ \Re  \ip{y_\theta}{c} -\lambda \norm{c}_2^2: \; \begin{bmatrix}
Q &c\\
c^* & 1
\end{bmatrix} \succeq 0, \qquad \sum_{i=1}^{N-j}Q_{i,i+j}=\begin{cases}
1 &j=0\\
0 &j=1,\ldots, N-1
\end{cases}, \quad c_{\setfreq^c} = 0
\right\}.}
\item[(II)]  In Step 2, let
\be{\label{eq:Delta_bar_robust}
\supsat_\subsetdir  := \bigcap_{\theta\in \subsetdir} \bigcup_{t \in \cT_\theta} \cN_\lambda \left( t \theta + \ell_\theta^\perp  \right).
}
where for the first $d$  directions in $\Theta'$, we let $\cN_\lambda(H) = H$ and for the remaining directions,
$$
\cN_\lambda(H) = \br{x:\; \mathrm{dist}(x,H)\leq \lambda}.
$$
\end{enumerate}
Note that in (II), the intersection of the hyperplanes directed by the first $d$ directions consist of $M^d$ points (see the proof of Lemma \ref{lem:holdae}), and intersection of this finite point set with the remaining hyperplanes ensure that we recover only the points which are within a small neighbourhood of the hyperplanes prescribed by the remaining directions.}

\rev{A detailed analysis of this reconstruction procedure is beyond the scope of this paper. Let us mention, however, that
\begin{itemize}
 \item the problem~\eqref{eq:blassodual1D} corresponds to the dual of the one-dimensional BLASSO problem. The study in~\cite{duval2015exact} (see also the discussion in Section~\ref{sec:constr_dual_cert}) ensures the support recovery at low noise for such problems, provided a Non-Degenerate Source Condition holds (which is empirically the case if the projection of all the spikes onto $\ell_\theta$ are sufficiently separated). As a result, for small values of $\lambda$ and $\norm{w}_2$, the saturations of 
 $p_\theta$ should be exactly the (slightly shifted) projections $\ip{\theta}{x_j}$, and nothing else ;
 
 \item in Figure \ref{fig:noise}, we present some numerical examples to demonstrate that the above modifications make Algorithm \algname ~  robust to small perturbations in the Fourier measurements.
\end{itemize}

In Figure \ref{fig:noise}, we show the reconstruction of $M$ point sources when given noisy samples along three directions. The samples which we observe are 
$$
y_{\text{observed}} = y_{\text{true}} + 0.15  \frac{n}{\norm{n}}, 
$$
where $n$ is a vector whose entries are i.i.d. normal Gaussian, and
$
y_{\text{true}} = \Phi \mu_0
$ with $$\Theta = \br{(\sin(\pi t),\cos(\pi t)): t \in \{0,1/3, 2/3\} },\quad \Gamma =\{-N,\ldots, N\}, \quad N= \lceil 1/\nu_{\min}\rceil.$$
In our examples, the  error on the recovered positions ($\mathrm{Err}_{\mathrm{pos}}$) and relative error on the recovered amplitudes ($\mathrm{Err}_{\mathrm{amp}}$) are all below 0.05.}

\rev{
\begin{remark}
If one is interested in the true BLASSO~\eqref{eq:primalnoisy}, the splitting approach exposed in this paper still provides some information in the sense that it constructs a dual certificate for~\eqref{eq:tvmin}. 
In the study of convex variational problems, it has been shown that the existence of a dual certificate $v$ implies robustness to noise corrupted measurements,  see \cite{burger2004convergence} for stability estimates with respect to the Bregman distance and \cite{candes2013super} for stability estimates with repect to the $L^2$ norm. In particular, these estimates depend on the decay of $\abs{v}$ away from 1 outside the support of $\mu_0$. One of the contributions of this paper is the construction of a dual certificate in the case radial lines sampling. We therefore expect that by careful analysis of our splitting certificate, one can apply the results of \cite{burger2004convergence} and \cite{candes2013super} to understand the robustness properties of \eqref{eq:primalnoisy}.
\end{remark}
}

\begin{figure}[ht]
\begin{center}
\begin{tabular}{c@{\hspace{0pt}}c@{\hspace{0pt}}c}
 \includegraphics[width = 0.3\textwidth,trim={5cm 1cm 5cm 1cm},clip]{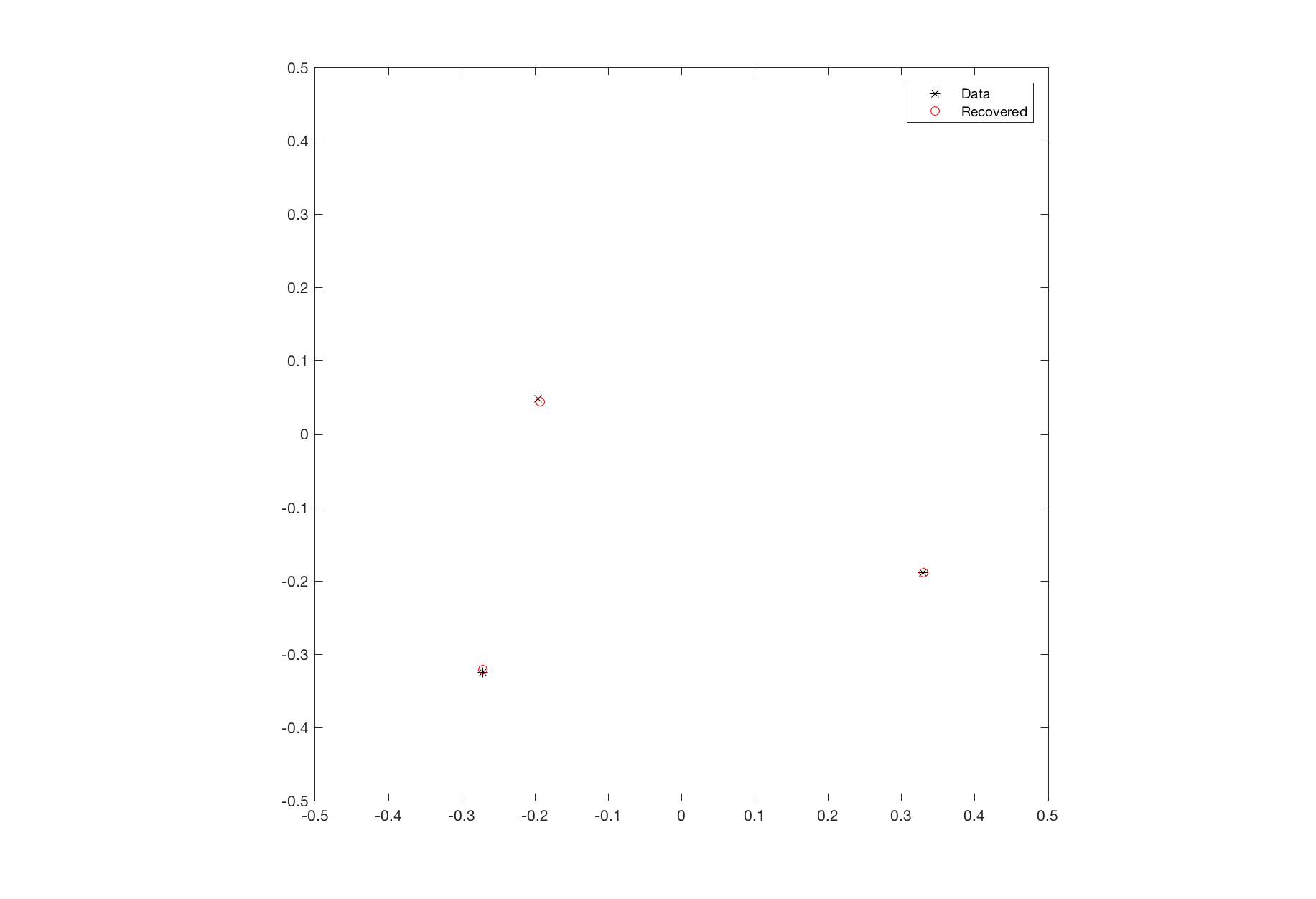}
& \includegraphics[width = 0.3\textwidth,trim={5cm 1cm 5cm 1cm},clip]{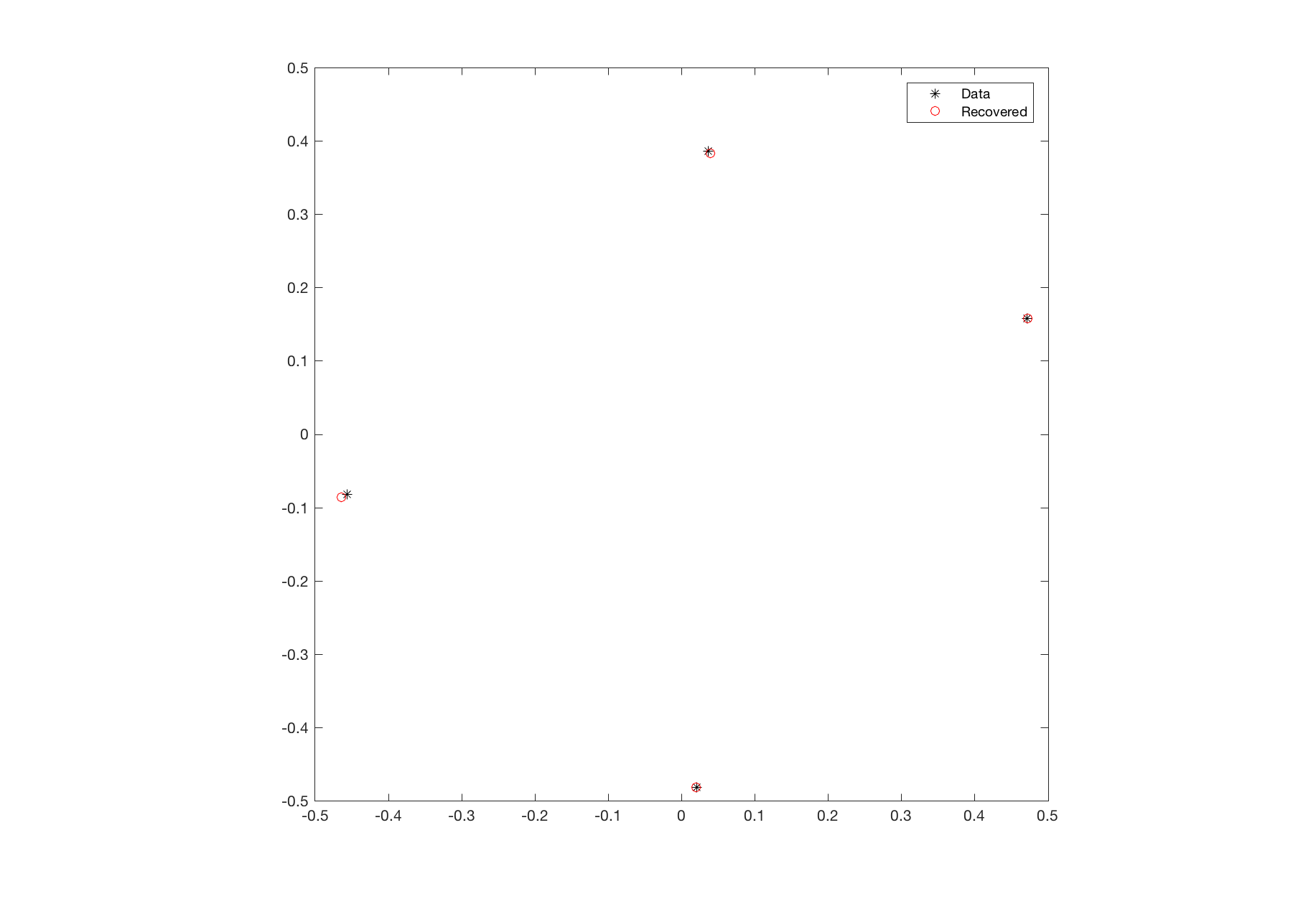}
& \includegraphics[width = 0.3\textwidth,trim={5cm 1cm 5cm 1cm},clip]{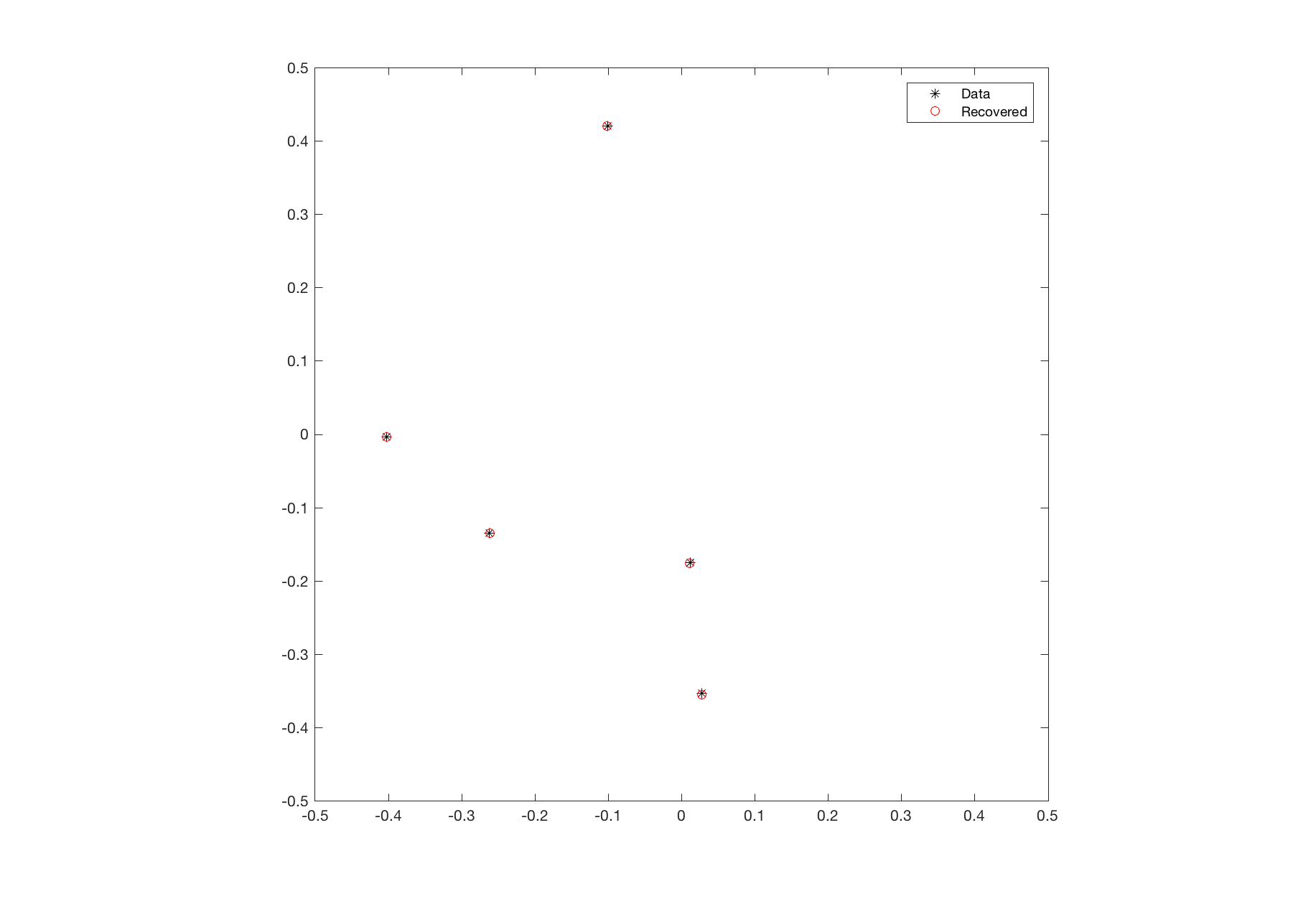}\\
 \includegraphics[width = 0.32\textwidth,trim={2cm 1cm 1cm 1cm},clip]{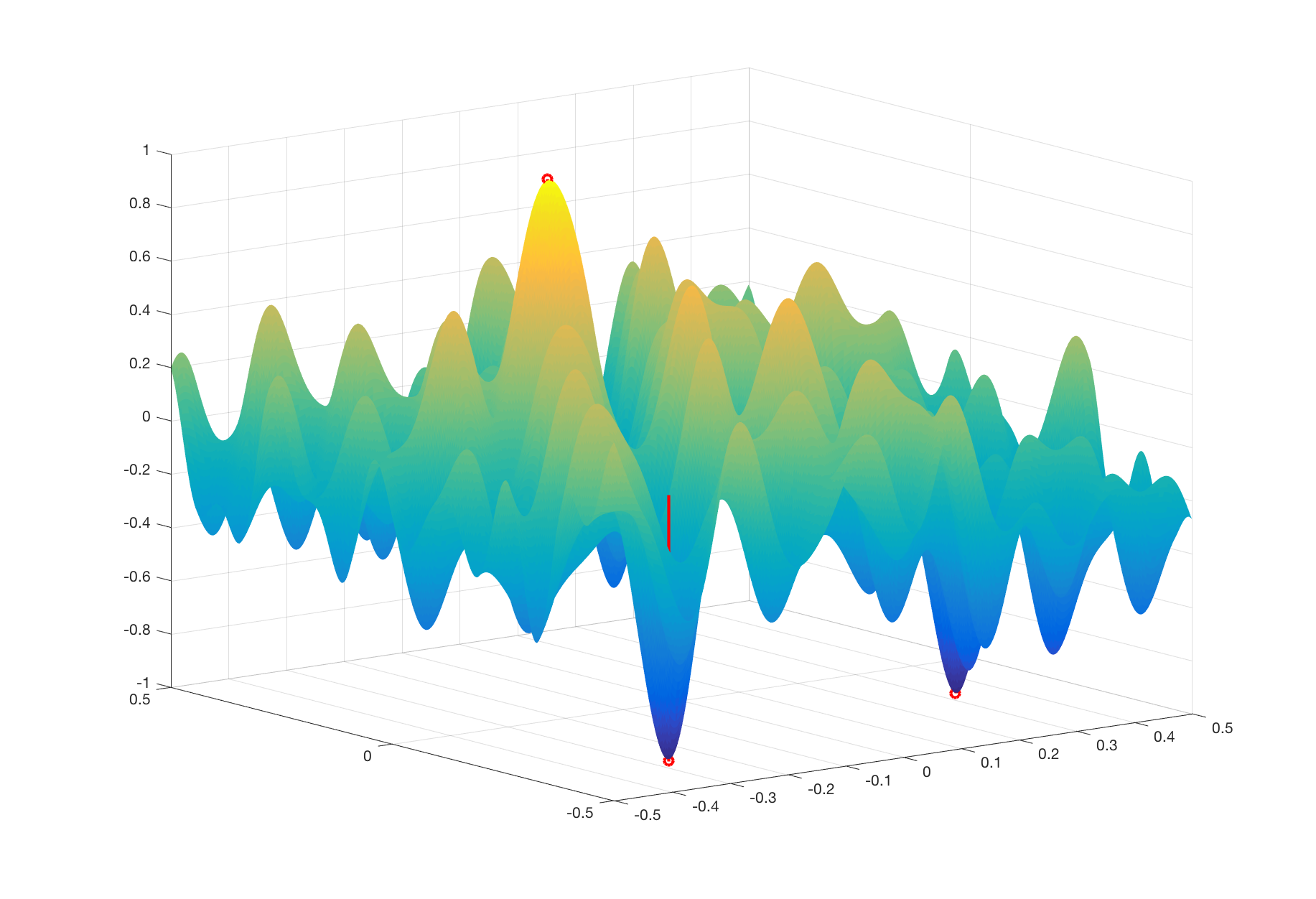}
& \includegraphics[width = 0.32\textwidth,trim={2cm 1cm 1cm 1cm},clip]{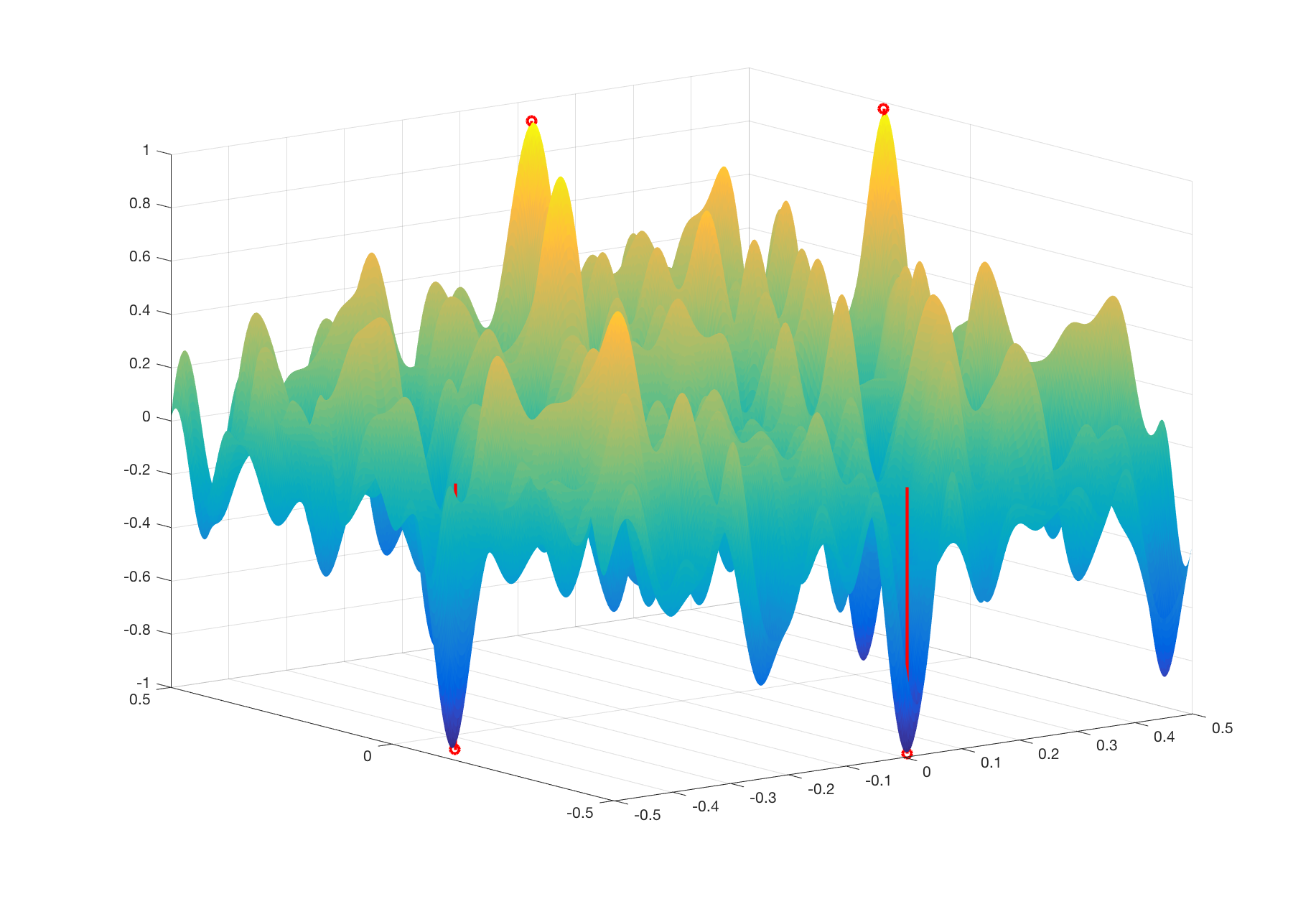}
& \includegraphics[width = 0.32\textwidth,trim={2cm 1cm 1cm 1cm},clip]{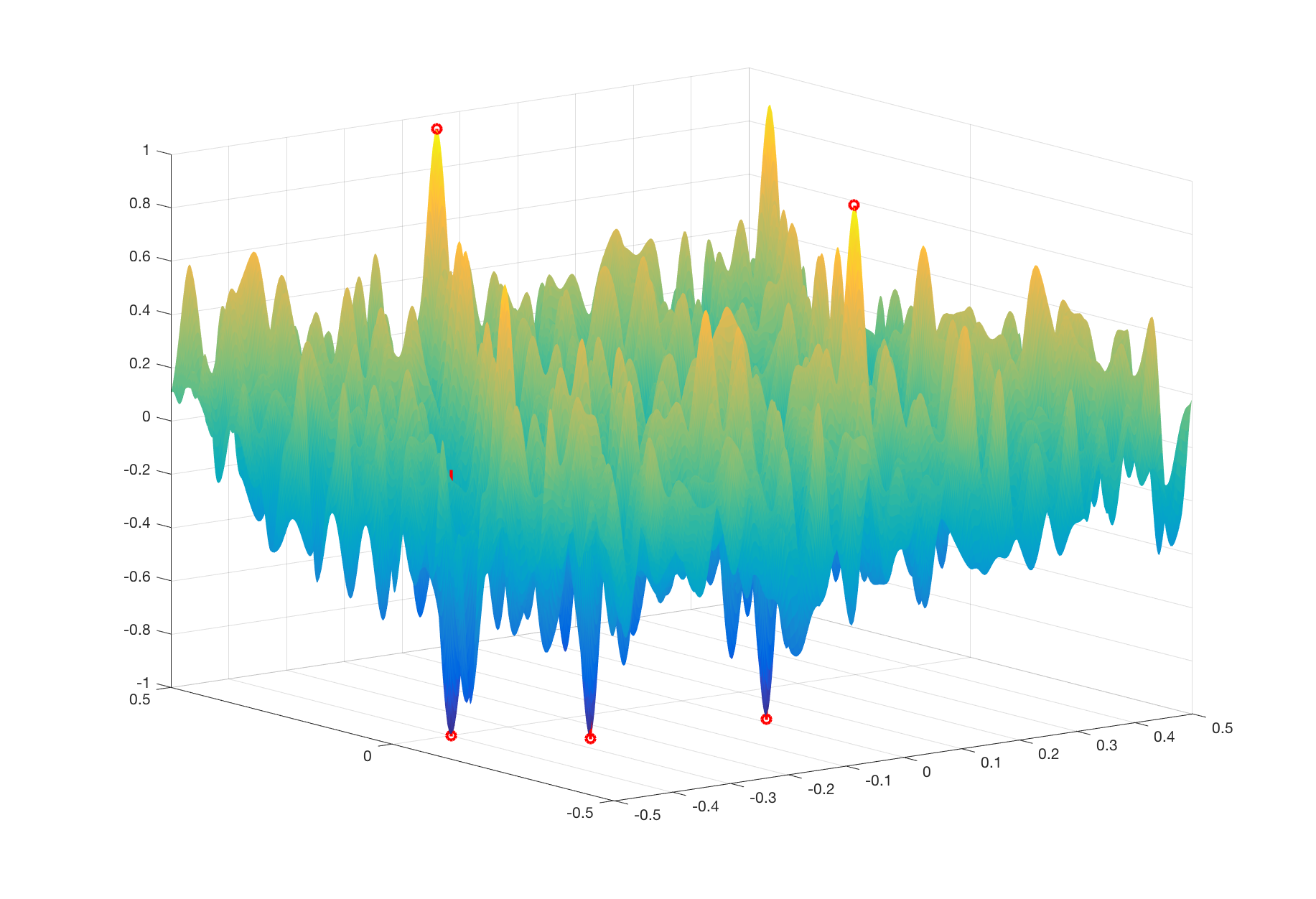}\\
  \includegraphics[width = 0.3\textwidth,trim={5cm 1cm 5cm 1cm},clip]{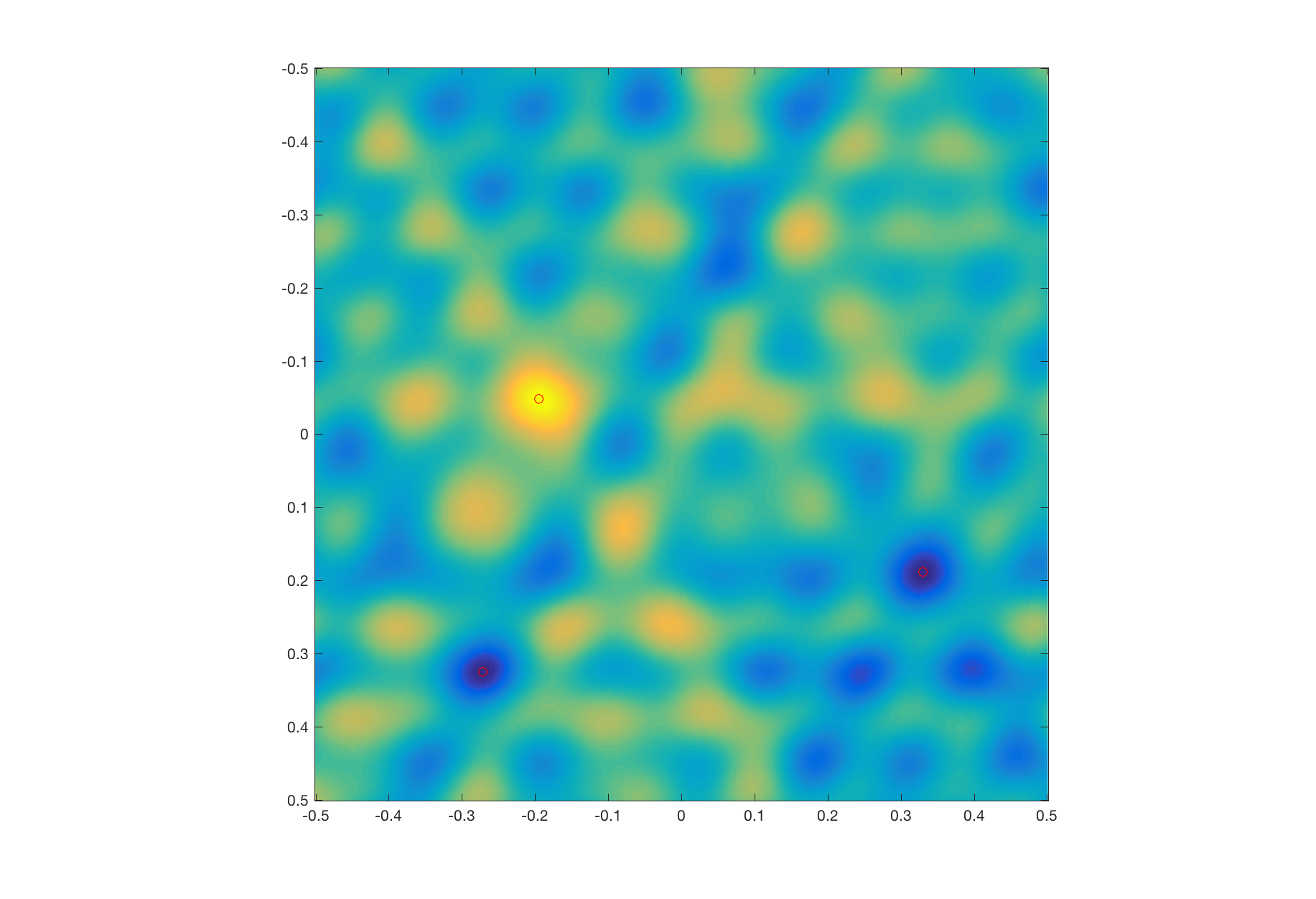} 
 & \includegraphics[width = 0.3\textwidth,trim={5cm 1cm 5cm 1cm},clip]{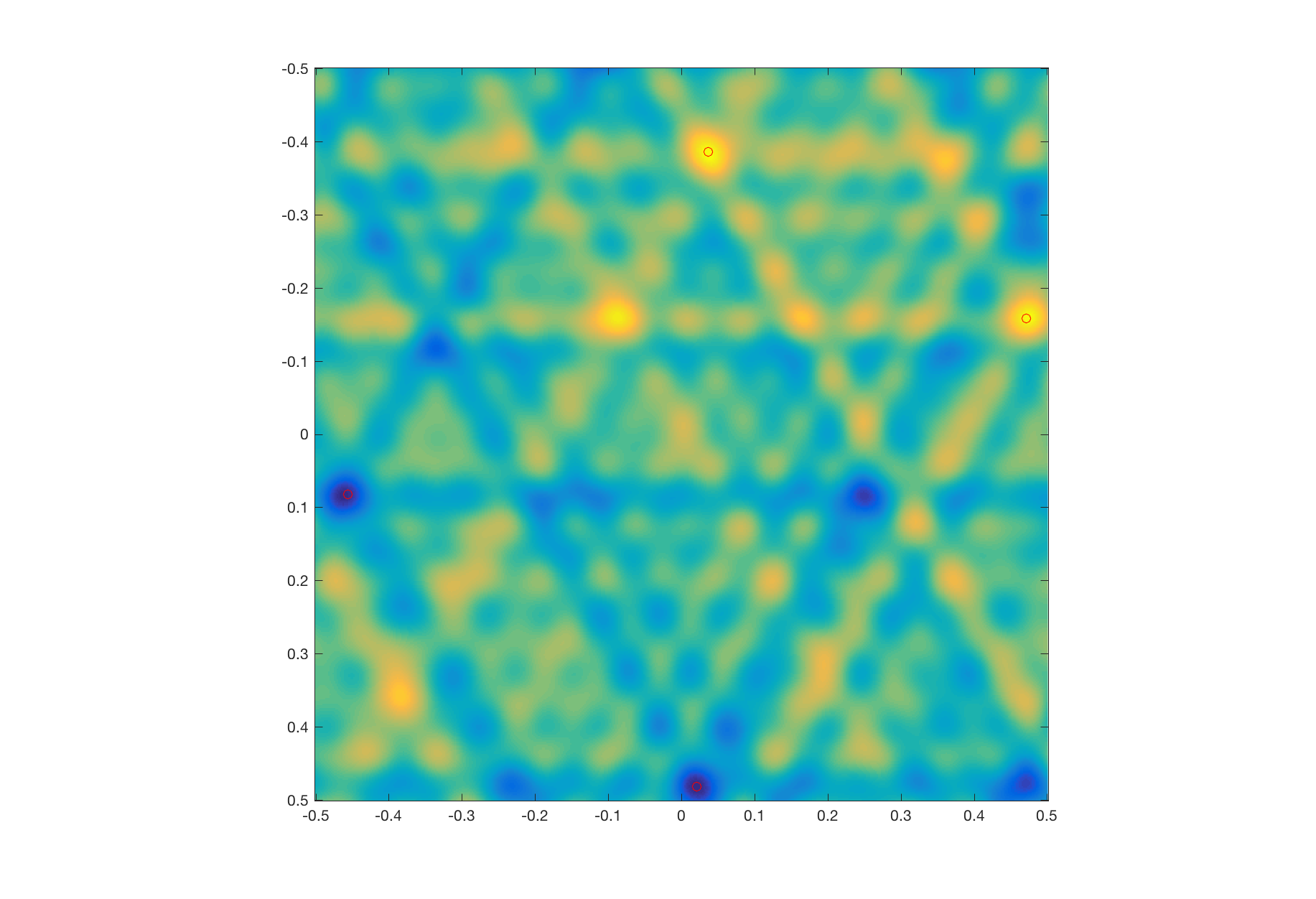}
& \includegraphics[width = 0.3\textwidth,trim={5cm 1cm 5cm 1cm},clip]{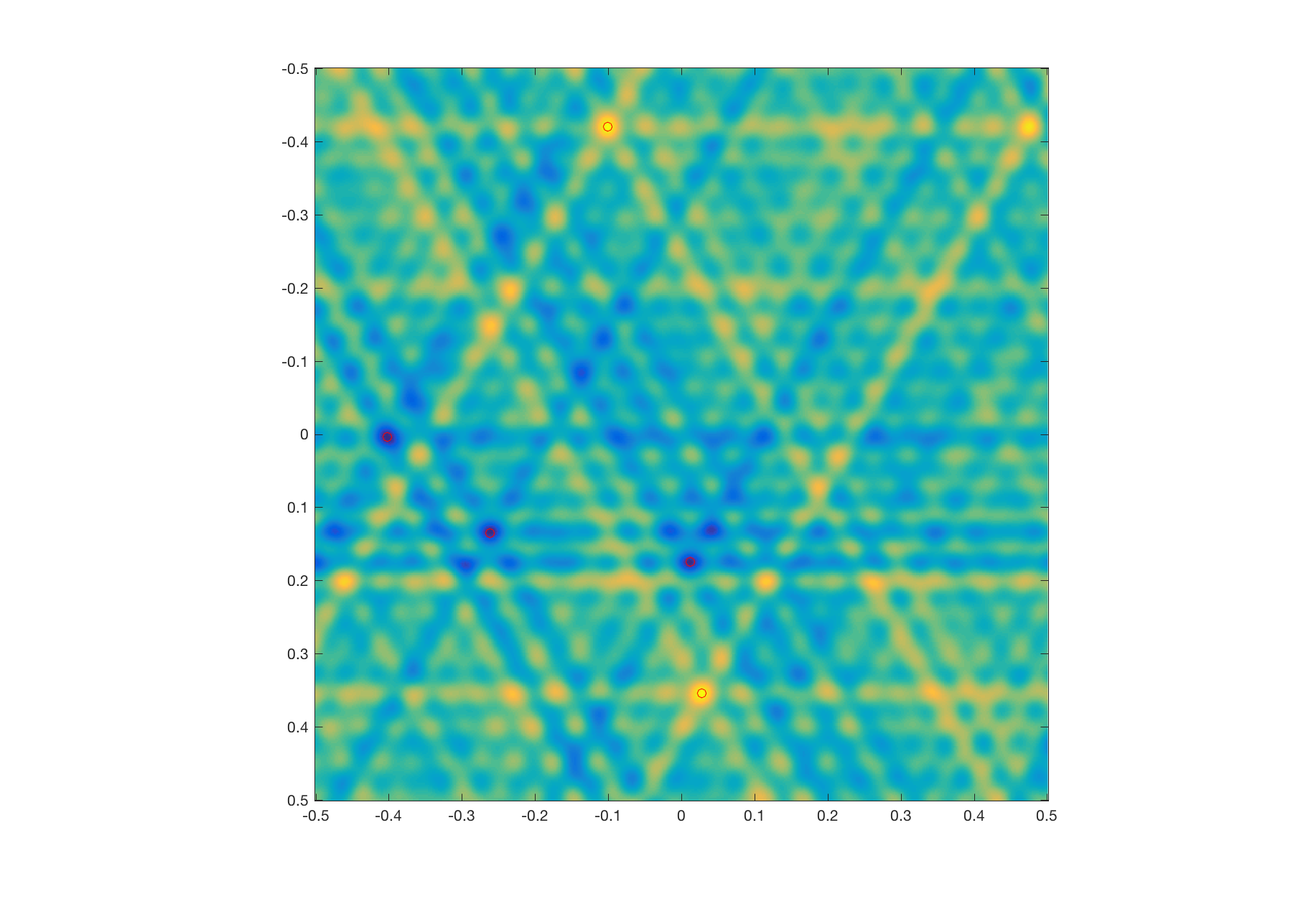}
\end{tabular}
\end{center}
\caption{\rev{Reconstruction of point sources when the Fourier measurements have a relative error of 0.15. The top row shows the original and reconstructed positions, the middle row shows the corresponding dual certificates (with the true positions shown in red) and the bottom row shows a 2-dimensional view of the dual certificates. Left column: $M=3$, $\nu_{\min} = 0.1208$, $\mathrm{Err_{pos} = 0.0061}$, $\mathrm{Err_{amp} = 0.0359}$. Middle column: $M=4$, $\nu_{\min} = 0.0704$, $\mathrm{Err_{pos} = 0.0098}$, $\mathrm{Err_{amp} = 0.0280}$. Right column: $M=5$, $\nu_{\min} = 0.0395$, $\mathrm{Err_{pos} = 0.0021}$, $\mathrm{Err_{amp} = 0.0261}$. \label{fig:noise}}}
\end{figure}

\section{Other related works and further comments}\label{sec:literature}

\subsection{Prony methods} 
The question of how many projections are required to recover $M$ sources has been of interest since the introduction of the Radon transform  by J. Radon in 1917 \cite{radon20051}. 
In particular, it is long known \cite{Renyi} that $M$ points are completely determined if its projection along $M+1$ lines are known. Since the early 2000's there has been numerous works describing how one can compute the parameters related to $M$ point sources using Prony's method, when sampling its Fourier transform along $M+1$ (or fewer) radial lines.
In \cite{vetterli_radon}, this problem was formulated in the framework of signals with a finite rate of innovation \cite{vetterli2002sampling} and it was shown, using Prony's method (also referred to as the annihilating filter method), that one can exactly recover the parameters of $M$ point sources from $2M$ samples along $M+1$ lines.  In \cite{plonka2013many}, it is shown that one can recover the parameters of $M$ point sources from $2M$ samples along 3 lines, which are adaptively chosen. This result is similar to that of \cite{potts2013parameter}, where the authors present an algorithm to sample only along sufficiently many lines. In these works, the  algorithms rely on the fact that the multi-dimensional parameter estimation problem can be projected onto a series of one-dimensional problems, under which one can simply apply the  Prony's method in the univariate setting.

\subsection{\revision{Links to the} Radon transform}
Here, we provide an informal discussion of the links between the results of this article and the Radon transform.

The Radon transform of a measure $\mu_0$, which consists of the collection of the projections  of the measure $\mu_0$ onto the line directed by $\theta \in \bbS^{d-1}$. That is, for $\theta\in \bbS^{d-1}$, we observe the measure $\projU_\sharp \mu\in \cM(\ell_\theta)$, where $\projU_\sharp \mu$ is the pushforward of $\mu$ by $P_{\ell_\theta}$.  
So, given any bounded continuous function $f:\ell_\theta\rightarrow \bbR$,
\begin{align}
  \int_{\ell_\theta} f(x')\diff \projU_\sharp \mu(x')&=\int_{\bbR^d}f(\projU(x))\diff \mu(x).
\end{align}
For example, in the particular case where $\mu$ has density $\rho\in L^1(\bbR^d)$ with respect to the $d$-dimensional Lebesgue measure $\cL^d$, then $\projU_\sharp \mu$ has density
\begin{equation*}
  \rho_\theta(s)= \int_{\ell_\theta^\bot} \rho(s\theta +v)\diff\cL^{d-1}(v). 
\end{equation*}
with respect to the Lebesgue measure $\cL^1$ on $\ell_\theta$, where $\ell_\theta^\bot$ denotes the orthogonal space to $\ell_\theta$.

The Fourier slice theorem states that the slice of $\cF \mu$ in the direction $\theta$ is the one-dimensional Fourier transform of $\projU_\sharp \mu$, so, for all $t\in\bbR$,
\begin{align*}
 \cF \mu(t\theta)&=\int_{\bbR^d}e^{-2\imath\pi\ip{x}{t\theta}}\diff \mu(x)
                                   =\int_{\bbR^d}e^{-2\imath\pi\ip{\projU(x)}{t\theta}}\diff \mu(x)
                                   =\int_{\ell_\theta}e^{-2\imath\pi t \ip{x'}{\theta}}\diff \projU_\sharp \mu(x').
\end{align*}
Consequently, given the Radon transform of $\mu_0\in \cM(X)$ along a finite number of directions $\Theta$, by letting $\eta_\theta = P_{\ell_\theta} \sharp \mu_0$,
$$
\min_{\mu\in\cM(X)} \norm{\mu}_{TV} \text{ subject to } P_{\ell_\theta} \sharp \mu = \eta_\theta, \quad \forall ~ \theta\in\Theta,
$$
is equivalent to 
$$
\min_{\mu\in\cM(X)} \norm{\mu}_{TV} \text{ subject to } (\cF \mu)(k\theta) = ( \cF \mu_0)(k\theta), \quad \forall ~ \theta\in\Theta, \quad \forall ~k\in\bbZ
$$
(since $\domain=\overline{B}(0,1/2)$, there is no aliasing artifact due to the $1$-periodization of $P_{\ell_\theta} \sharp \mu_0$). We finally remark that in practice, the presence of white noise will typically render the higher Fourier coefficients unusable. Therefore, one is led to consider the above minimization problem with $\abs{k}\leq N$
for some $N\in\bbN$. This is precisely the problem studied in this paper and has been considered for applications such as electron tomography \cite{leary2013compressed}.

\subsection{Generalizations \revision{to shift invariant linear operators}}\label{sec:generalizations}

Suppose that $\Phi$ is a Fourier sampling operator, let  $\cS(X)$  denote the Schwartz space and let $\cL:\cS(X)\to \cM(X)$ be any differential operator. Let $$\cM_\cL(X) = \br{f\in\cS(X): \norm{\cL f}_{TV} = \sup_{\varphi\in C_c^\infty(X): ~\norm{\varphi}_\infty = 1} \Re\ip{\cL f}{\varphi}<\infty}.
$$
We now consider the generalized total variation minimization problem, given measurements $y_0 = \Phi f_0$ for some $f_0\in\cM_\cL(X)$:
\be{\label{eq:gtv}
\inf_{f\in \cM_\cL} \norm{\cL f}_{TV} \text{ subject to } \Phi(f) = y_0.
}
This problem is a specialized form of the generalized total variation minimization problems recently considered in \cite{unser2016splines}, where $\cL$ is a linear shift invariant operator, which is associated with a finite dimensional null space  $\cN(\cL)$ and a Green's function $\rho \in \cM_\cL$ such that $\cL \rho = \delta$, the Dirac measure. Examples of differential operators for \eqref{eq:gtv} include the Laplacian and the fractional Laplacian, and the associated Green's functions are  polyharmonic splines \cite{madych1990polyharmonic}.

The main result of \cite{unser2016splines} establishes that, under mild conditions on the sampling operator $\Phi$ of rank $N$, any solution to \eqref{eq:gtv} is necessarily of the form
$$
f(x) = \sum_{k=1}^K \alpha_k \rho (x-x_k) + \sum_{n=1}^{N_0} \beta_n g_n(x),
$$
where $\{g_n:n=1,\ldots, N_0 \}$ is a basis for $\cN(\cL)$, $\{x_k\}_{k=1}^K\subset \bbR^d$, $\alpha\in\bbR^K$ and $\beta\in \bbR^{N_0}$, with $K\leq N-N_0$.
In particular, any solution to \eqref{eq:gtv} satisfies $\sum_{j=1}^K a_j \delta_{x_j}$. Now, returning to the case where $\Phi$ is a Fourier operator, given $f_0$  such that $\cL f_0 = \sum_{j=1}^M a_j \delta_{x_j}$, suppose that the following hold.
\begin{enumerate}
\item There exists a trigonometric polynomial $p= \Phi^* q$ such that
$$
p(x_j) = \sgn(a_j), \qquad \norm{p}_\infty \leq 1,
$$
and the extremal points of $p$ form a discrete set $\{\supx_j\}_{j=1}^{\tilde M}$ such that
$$
c\in \bbC^{\tilde M}\mapsto \Phi\left( \sum_{j=1}^{\tilde M} c_j \delta_{\supx_j}\right)\qquad \text{ is injective.}
$$
\item $\Phi$ is injective on $\cN(\cL)$, the null space of $\cL$.
\end{enumerate}
Then, one can show that $f_0$ is the unique minimizer to \eqref{eq:gtv}. Therefore, to understand the recoverability conditions on  the sampling operator $\Phi$, one simply needs to understand the construction of the appropriate interpolating polynomial, and the behaviour of $\Phi$  on $\cN_\cL$.  Therefore, for the Fourier operator restricted to radial lines, the techniques developped in this paper should be readily extendible for the analysis of \eqref{eq:gtv}.

%

\section{Proofs}\label{sec:proofs}
\label{sec:recovery}

\subsection{Preliminary results}

 We first require a definition.

\begin{definition}
Let $\mu \in \cM(\domain)$. Given $\theta \in \bbS^{d-1}$, let the one-dimensional subspace $\ell_\theta = \{ t \theta : t\in \bbR\}$ denote its corresponding radial line.
We say that $\theta \in \bbS^{d-1}$ is a projecting direction for $\mu$ if the projection  
$P_{\ell_\theta}: \domain \to \ell_\theta$, $ x \mapsto \ip{x}{\theta}\theta$ is injective on $\supp{\mu}$.
\end{definition}

Observe that in order for (A1) to make sense, for any values of $\sgn(a_j)$, it is necessarily the case that each element $\theta$ in $\setdir$ is a projecting direction for $\mu_0$. Note also that if the extremal points of each $p_\theta$ in (A1) is precisely $\{x_1,\ldots, x_M\}$, then the set $\supsat$ in \eqref{eq:Delta_bar} can be written as
\be{\label{eq:Delta_bar2}
\supsat = \bigcap_{\theta \in \Theta}\bigcup_{j=1}^M (x_j + \ell_\theta^\perp).
}
The following lemmas describe conditions under which this set is precisely the original positions $\{x_j:j=1,\ldots, M\}$.

\begin{lemma} \label{lem:holdae}
\begin{itemize}
\item[(i)] For almost every finite choice of $\setdir \subset \bbS^{d-1}$, each element $\theta$ in $\Theta$ is a projecting direction for $\mu_0$,
\item[(ii)] Let $\supsat$ be as defined in \eqref{eq:Delta_bar2}.  If  $\abs{\Theta}=d$, then $\supsat$ is a discrete set consisting of at most $M^d$ elements. Moreover,  for almost every choice of $\setdir \subset \bbS^{d-1}$ with $\abs{\Theta}\geq d+1$,  $\supsat  = \{x_1,\ldots, x_M\}$.
\end{itemize}
 
\end{lemma}
\prf{
  We first consider (i). A direction $\theta \in \bbS^{d-1}$ is non-projecting for $\mu_0$ if and only if there exists $x_j\neq x_k$ such that $x_j-x_k\in \ell_\theta^\perp$ or equivalently, $\theta \in \ell_{x_j-x_k}^\perp$. So,  $\theta \in \bbS^{d-1}$  is a non-projecting direction if and only if it belongs to one of the $M(M-1)/2$ hyperplanes $\br{\ell_{x_j-x_k}^\perp}_{1\leq j<k\leq M}$. However, since $ \bbS^{d-1} \cap \left(\bigcup_{1\leq j<k\leq M}\ell_{x_j-x_k}^\perp\right)$ is of measure 0 (with respect to the Hausdorff measure on $\bbS^{d-1}$), (i) is satisfied for almost every choice of $\setdir$ from $\bbS^{d-1}$.

  For  (ii), observe that for almost every choice of $d$ unit vectors $\setdir':= \{\theta_1,\ldots, \theta_d\}$, $\setdir'$ spans $\bbR^d$. Then $\bigcap_{k=1}^d \bigcup_{j=1}^M (x_j + \ell_{\theta_k}^\perp)$ consists of at most $M^d$ points. Indeed, let $z\in\bbR^d$. Then $z\in\bigcap_{k=1}^d \bigcup_{j=1}^M (x_j + \ell_{\theta_k}^\perp)$ if and only if for $1\leq k \leq d$, there exists $j_k \in \{1, \ldots,M\}$ such that $z\in x_{j_k}+  \ell_{\theta_k}^\perp$. That is equivalent to $\ip{\theta_k}{z} =\ip{\theta_k}{x_{j_k}}$ for $1\leq k \leq d$, or equivalently
\begin{equation}\label{eq:innerprod}
  \begin{pmatrix}
    \theta_1 & \theta_2 &\cdots & \theta_d
  \end{pmatrix}^T     z
=  b\end{equation}
for some $b\in \bbR^d$ such that  $b_{k}:= \ip{\theta_k}{x_{j_k}}$. There are at most $M^d$ possible choices of $b$, and for each $b$, there is a unique solution $z$ to the system~\eqref{eq:innerprod}.

Now, let us denote by $Z$ the set of all such $z$ (so that $\abs{Z}\leq M^d$). So, $\setdir' \cup \{\theta_{d+1}\}$ satisfies (ii) provided that there does not exist $z\in Z\setminus \{x_1,\ldots x_M\}$ such that
$z \in \bigcup_{i=1}^M (x_j + \ell_{\theta_{d+1}}^\perp)$. This is equivalent to choosing $\theta_{d+1}$ such that 
$$
\theta_{d+1} \in \bbS^{d-1} \setminus H, \qquad H := \bigcup\br{\ell^\perp_{x_j - z}: 1\leq j \leq M, z\in Z\setminus\{x_1,\ldots,x_M\} }.
$$
Observe that $H$ is the union of finitely many hyperplanes and so, just as in \rev{(i)}, $H\cap \bbS^{d-1}$ is of measure 0 and almost every choice of $\theta_{d+1}$   will ensure that $\setdir' \cup \br{\theta_{d+1}}$ satisfies (ii). In particular, if $\abs{\setdir}=d+1$, then (ii) is satisfied for almost every choice of $\setdir$. \rev{The conclusion thus follows by observing that $\{x_1,\ldots, x_M\}\subset \tilde \Delta$ for any choice of $\Theta$.}

}

The next lemma deals with the converse situation where the set of directions $\setdir$ is fixed but the positions $\{x_1,\ldots, x_M\}$ are random.

\begin{lemma}\label{lem:indep}
Let $\Theta \subset \bbS^{d-1}$ be a set of cardinality $L\geq d+1$. Assume that any $d$ elements of $\Theta$ are linearly independent.  Then for almost every choice  of $M$ points $\br{x_1,\ldots, x_M}$ in $X$, $\supsat = \{x_1,\ldots, x_M\}$.
\end{lemma}
\begin{remark}
In the case of  $d=2$, any 2 distinct elements of $\bbS^1$ will be linearly independent. In general, if we let $0 <t_1<t_2<\cdots <t_L$, then any $d$ columns of the matrix
$$
\begin{pmatrix}
1 & 1 &\cdots &1\\
t_1 & t_2& \cdots &t_L\\
\vdots &\vdots& & \vdots\\
t_1^{d-1} & t_2^{d-1} & \cdots & t_L^{d-1}
\end{pmatrix}
$$
forms a Vandermonde matrix and hence forms a linearly independent set. So, an example of valid set $\Theta$ for Lemma \ref{lem:indep} is the set $\{ u_j:~j=1,\ldots, L\}$ with
$$
 u_j := v_j/\norm{v_j}_2\qquad v_j := (1,t_j,t_j^2,\cdots, t_j^L).
$$
\end{remark}

\begin{proof}
Suppose we are given $n$ points $\br{x_1,\ldots, x_n}$ such that
\be{\label{eq:assump_lem_indep}
\bigcap_{k=1}^L \bigcup_{j=1}^n (x_j + \ell^\perp_{\theta_k}) = \br{x_1,\ldots, x_n}.
}
Note that this is automatically true if $n = 1$.
Let $x_{n+1}$ be a random variable in $X$. Let $P_{n+1}$ be the probability that
\bes{
\bigcap_{k=1}^L \bigcup_{j=1}^{n+1} (x_j + \ell^\perp_{\theta_k}) = \br{x_1,\ldots, x_{n+1}}.
}
 Suppose that there exists
\be{\label{eq:a3}
z\in \bigcap_{k=1}^L \bigcup_{j=1}^{n+1} (x_j + \ell^\perp_{\theta_k}) \setminus \br{x_1,\ldots, x_{n+1}}.
}
Then, there exists $\theta\in \Theta$ such that $\ip{z}{\theta} = \ip{x_{n+1}}{\theta}$, otherwise, we would violate assumption \eqref{eq:assump_lem_indep}. Also, since $z\neq x_{n+1}$, there exists $\subsetdir \subset \Theta$ of cardinality at least $L-d+1$ such that
\be{\label{eq:a2}
z\in \bigcap_{\theta\in\subsetdir} \bigcup_{j=1}^n(x_j+\ell_\theta^\perp).
}
Indeed, if this was not the case, then there exists $\{\theta_{k_1},\ldots,\theta_{k_d}\}\subset \Theta$ such that
$$
\ip{z- x_{n+1}}{\theta_{j_k}} = 0, \qquad k=1,\ldots, d.
$$
Since any $d$ elements of $\Theta$ form a basis of $\bbR^d$, this would imply that $z = x_{n+1}$ and thus contradicting \eqref{eq:a3}.

Let $\subsetdir\subseteq \Theta$ be the largest subset for which \eqref{eq:a2} holds. Note that  $\abs{\subsetdir} = d' \in [2,d]$. Then,
\be{\label{eq:a4}
\ip{z-x_{n+1}}{\theta} = 0 \qquad \forall ~ \theta \in \Upsilon:= \Theta\setminus \subsetdir.
}
From \eqref{eq:a2}, we see that $z$ belongs to the intersection of $d'$ hyperplanes. In particular, it belongs to a subspace $V$ of dimension $d-d'\leq d-2$. Also, from \eqref{eq:a4}, we have that $x_{n+1} -z \in W^\perp$, where $W = \mathrm{span}(\theta\in \Upsilon)$. Note that $\Upsilon$ contains $L-d'\geq d+1-d'$ elements. So, $W$ is a subspace of dimension at least $d+1-d'$ and $W^\perp$ is of dimension at most $d'-1$. So, $x_{n+1} \in V+W^\perp$ is contained in a subspace of dimension at most $d-1$, which is a set of measure 0 in $X$. Also, since there are finitely many combinations of sets of cardinality $d'\in [2,d]$ in $\Theta$, and the  union of finitely many sets of zero measure is also of zero measure, \eqref{eq:a3} holds with probability 0. 
Therefore, $P_{n+1} = 1$ and by applying the chain rule for probabilities, the assertion of this lemma holds.
\end{proof}

\subsection{The minimum separation distance and the sampling range}

Solutions to \eqref{eq:tvmin} in the case $d=1$  have been extensively studied~\cite{de2012exact,candes2014towards}. These works established conditions based on the \textit{minimum separation distance}, under which one can construct one-variate trigonometric polynomials which interpolate the sign pattern related to the underlying measure.

\begin{lemma}\cite{tang2013compressed}\label{lem:dual_poly_1d}
  Let $\Delta = \{ t_j: j=1,\ldots, M\}\subset \bbT$ and suppose that $\mdist(\Delta) \geq 2/N$.  
  \begin{itemize}
  \item[(i)] Let $\{s_j\}_{j=1}^M\subset \br{x\in\bbC: ~ \abs{x}=1}$. If $\Gamma = \{-N,\ldots, N\}$, then the operator  $\Psi:\bbC^M \to \bbC^m$ with $\Psi a = \left( \sum_{j=1}^M a_j e^{-i2\pi k t_j}\right)_{k\in \Gamma}$ is injective and there exists $q = \sum_{j\in\Omega}\alpha_j e^{i2\pi j\cdot}$ such that 
$$
q(t_j) = s_j, \quad j=1,\ldots, M \quad\text{and}\quad \abs{q(t)}<1, \quad \forall t\not\in \Delta.
$$
\item[(ii)]  Let $\{s_j\}_{j=1}^M$ be drawn i.i.d. from the uniform
distribution on the complex unit circle. If $\Gamma \subseteq \{-N,\ldots, N \}$ consists of $m$ indices chosen uniformly at random with
$$
m\gtrsim \max \{\log^2(N/\delta), M\log(M/\delta)\log(N/\delta)  \},
$$
then the conclusions of (i) hold with probability at least $1-\delta$.
  \end{itemize}

\end{lemma}

As a consequence of Lemma \ref{lem:dual_poly_1d}, Lemma \ref{lem:holdae} and Proposition \ref{prop:split}, we can prove the first of our  main results.

\begin{proof}[Proof of Theorem \ref{thm:rand_angles}]
By Proposition \ref{prop:split}, it suffices to show that conditions (A1) and (A2) are satisfied.
We first consider the case where $\Gamma = \{-N,\ldots,N\}$. By the assumption on the minimal distance $\nu_{\min}$ and by Lemma \eqref{lem:dual_poly_1d}, condition (A1) holds for all $\theta\in\Theta$. Furthermore, since for each $\theta\in\Theta$, the extremal values of $p_\theta$ are precisely $\br{\ip{x}{\theta}:\theta\in\Theta}$, by Lemma \ref{lem:holdae}, we have that
\be{\label{eq:A2become}
\bigcap_{\theta \in\Theta}\bigcup_{j=1}^M (x_j + \ell_\theta^\perp) = \{x_1,\ldots, x_M\}
}
holds with probability 1. So, the injectivity requirement of (A2) is satisfied if there exists $\theta \in \setdir$ such that the operator $A_\theta : \bbC^M \to \bbC^\setfreq$ defined by
$$
A_\theta: \bbC^M \to \bbC^\setfreq, \quad
a\in\bbC^M \mapsto \left(\sum_{j=1}^M a_j e^{-i 2\pi k \ip{\theta}{x_j}}\right)_{k\in \setfreq}.
$$
is injective.  However, this is true since $\setfreq=\{-N,\ldots, N\}$ contains $M$ consecutive integers, and standard results on Vandermonde matrices imply that $A_\theta$ is injective.

In the case where $\Gamma$ consists of $m$ indices chosen at random from $\br{-N,\ldots, N}$, by applying Lemma \ref{lem:dual_poly_1d} and the union bound, (A1) holds for all $\theta\in\Theta$ with probability at least $1-(d+1)\delta$. By Lemma \ref{lem:holdae}, (A2) holds with probability 1 and by Lemma \ref{lem:dual_poly_1d}, given any $\theta\in\Theta$, $A_\theta$ is injective with  probability  least $1-\delta$. So, (A2) holds with probability at least $1-\delta$.

\end{proof}

\subsection{Estimation of the minimum separation distance}
Let $Z=\br{z_i:~ i=1,\ldots, M}$ be i.i.d. random variables on the $d$-dimensional torus $\bbT^d=(\bbR/\bbZ)^d$ with density law $f$.  Given $z$, $z'\in\bbT^d$, we denote by $\dtd{z-z'}$ the canonical (obtained from the Euclidean norm) distance on $\bbT^d$. Let
$E_M^\delta=\{\mdist(Z) \geqslant \delta\}$, where, as before, $\mdist(Z):=\min_{z,z'\in Z, z\neq z'}\dtd{z-z'}$ and let $V_d$ denote the volume of the unit ball of $\mathbb{R}^d$.

The following two lemmas \revision{show} that with high probability $\mdist \asymp M^{-2/d}$, with Lemma \ref{lemdistance} showing that $\mdist\gtrsim M^{-2/d}$ and Lemma \ref{lem:distance_ub} showing that $\mdist\lesssim M^{-2/d}$. 
\begin{lemma}\label{lemdistance}

For any $\rho>0$ with $\delta=\left(\frac{2\rho}{M(M-1)V_d \Vert f\Vert_{2}^2}\right)^{1/d}<1$,  we have
\begin{equation}
\bbP\left(E_M^\delta\right)\geqslant 1-\rho.
\end{equation}
\end{lemma}
\begin{proof}
We first remark that 
\begin{align}
\begin{split}\label{eq:PE2}
  \bbP\left(E_2^\delta\right)&=1-\bbP\left(\dtd{z_1-z_2}<\delta\right)=
  1-\iint_{\bbT^d\times \bbT^d} \chi_{\dtd{t_1-t_2}<\delta}f_{z_1}(t_1)f_{z_2}(t_2)\mathrm{d}t_1\mathrm{d}t_2\\
&=
  1-\iint_{\bbT^d\times \bbT^d} \chi_{\dtd{s}<\delta}f_{z_1}(t)f_{z_2}(t+s)\mathrm{d}t\mathrm{d}s\\
  &= 1-\int_{\bbT^d}\chi_{\dtd{s}<\delta}\left(\int_{\bbT^d}f_{z_2}(t+s)f_{z_1}(t)\mathrm{d}t \right) \mathrm{d}s\geqslant 1-\Vert f\Vert_{2}^2\delta^d V_d.
\end{split}
\end{align}
Then we observe that 
\begin{align*}
  \bbP\left(E_M^\delta\right)&=1-\bbP\left(\exists (i,j),\,i\leqslant M,\,j\leqslant M,\,i\neq j\text{ such that }\dtd{z_i-z_j}<\delta\right)\\
                             &\geqslant 1-\sum_{1\leqslant i<j\leqslant M} \bbP\left(\dtd{z_i-z_j}<\delta\right)
  \geqslant 1-\frac{M(M-1)}{2} \bbP\left(\dtd{z_1-z_2}<\delta\right)\\
&\geqslant 1-\frac{M(M-1)}{2}\Vert f\Vert_{2}^2\delta^d V_d=1-\rho.
\end{align*}
\end{proof}

\begin{lemma}\label{lem:distance_ub}
For any $t >0$, defining $\delta=2\left(\frac{t}{(M^2-M-1)V_d}\right)^{1/d}$, 
we have
\begin{equation}
\bbP\left(E_M^\delta\right)\leqslant e^{-t}
\end{equation}
\end{lemma}
\begin{proof}
We use now the expression of $\bbP(E_M^\delta)$ as a product of conditional probabilities :
\begin{equation}
\bbP\left(E_M^\delta\right)=\prod_{j=3}^M \bbP\left(E^\delta_{k}|E^{\delta}_{k-1}\right)\times \bbP(E^\delta_2)
\end{equation}
and  recall from \eqref{eq:PE2} that
\begin{equation}
\bbP(E^\delta_2)=1-V_d\delta^d\leqslant e^{-V_d\delta^d}.
\end{equation}
We observe then that if $E_{k-1}^\delta$ is satisfied, all the balls centered in $z_i,\, (i\leqslant k-1)$  with radius $\frac{\delta}{2}$ are disjoint, hence 
the volume of the union of balls centered in $z_i$ for $i\leqslant k-1$ with radius $\delta$ is at least $\frac{(k-1)V_d\delta^d}{2^d}$. This implies that 
\begin{equation}
\bbP\left(E^\delta_{k}|E^{\delta}_{k-1}\right)\leqslant \max\br{0, 1-\frac{(k-1)V_d\delta^d}{2^d}}\leqslant e^{-\frac{(k-1)V_d\delta^d}{2^d}}.
\end{equation}
So,
\begin{equation}
\bbP\left(E_M^\delta\right)\leqslant\prod_{k=3}^Me^{-\frac{(k-1)V_d\delta^d}{2^d}}\leqslant e^{-\frac{(M^2-M-1)V_d\delta^d}{2^d}}=e^{-t}.
\end{equation}
\end{proof}

\subsection{Proof of Theorems \ref{thm2}}

Using Lemma \ref{lemdistance} we can prove Theorem \ref{thm2}:

\prf{[Proof of Theorem \ref{thm2}]
By Proposition \ref{prop:split}, it suffices to show that conditions (A1) and (A2) are satisfied. For the first part, 
 we shall apply Lemma \ref{lemdistance} to show that for our $N$, $\nu_{\min}\geq 2/N$  with probability larger than $1-(d+1)\delta$. 
Hence, we need to consider  the laws of the projections of points $x_j$ on each directions $\theta\in \Theta$. Since the original density law of each $x_j$ is uniform on the $\ell_2$-ball $X$ of radius $1/2$, if we denote by $f_{d,X}$ the density law of the projection $(t_j)_{j=1}^d \mapsto (t_1,0,\cdots,0)$,  $f_{d,X}(t)$ is the normalized $\cH^{d-1}$ measure of the intersection of the hyperplane $\{(t,x): x\in\bbR^{d-1}\}$ with  the ball $X$, that is a $(d-1)$-dimensional ball of radius $\sqrt{1/4-t^2}$. Therefore,
$$
f_{d,X}(t) = \frac{V_{d-1}(\sqrt{1/4-t^2})}{V_d(1/2)} = 2^d(1/4-t^2)^{(d-1)/2} \frac{V_{d-1}(1)}{V_d(1)},
$$
where $V_{d}(r)$ denotes the volume of the $d$-dimensional ball of radius $r$. Hence,
\begin{align*}
\norm{f_{d,X}}_2^2&= 4\left(\frac{V_{d-1}(1)}{V_d(1)}\right)^2
\int_0^1 (1-t^2)^{d-1}\mathrm{d}t =
\frac{2V_{d-1}(1)^2V_{2d-1}(1)}{V_d(1)^2 V_{2d-2}(1)}\\
&=\frac{2\pi^{-\frac{1}{2}}\mathrm{G}(\frac{d}{2}+1)^2 \mathrm{G}(d)}{\mathrm{G}(\frac{d+1}{2})^2 \mathrm{G}(d+\frac{1}{2})}\leqslant \frac{2d+2}{\sqrt{\pi(2d-1)}},
\end{align*}   
where $\mathrm{G}$  denotes the Gamma function. Observe that we should compute the density of the $1$-periodized projections of the $x_i$'s, hence we should compute the $1$-periodization of $f_{d,X}$. But the assumption on the domain ($\domain=\overline{B}(0,1/2)$) prevents aliasing effects , and the periodization of $f_{d,X}$ coincides with its restriction on $(-1/2,1/2]$.

Thus, we can lower bound the minimum separation distance along each projected direction.  Then, by applying (i) of Lemma \ref{lem:dual_poly_1d} and the union bound, when $\Gamma = \br{-N,\ldots, N}$,  we can conclude that (A1) holds with $L'=L$ with probability at least $1-(d+1)\delta$. 
 For (A2), note that the fact that $(x_j)_{j=1}^M$ are random and independent of the directions $\setdir$ ensures, by Lemma \ref{lem:indep}, that we simply need to assert the injectivity condition in (A2) on $\supsat = \{x_1,\ldots, x_M \}$. However, (A2) is true by standard results on Vandermonde matrices.

The second part of Theorem \ref{thm2} is proved similarly, by applying Lemma \ref{lemdistance}, the union bound and (ii) of Lemma \ref{lem:dual_poly_1d}.

}

\section{Conclusion}
There has been several theoretical works on the recovery of point sources (in an infinite dimensional setting) via convex optimization approaches. Previous works tend to consider sampling on some uniform grid, whereas, this article presents the analogous results in the case where we are restricted to sampling along radial lines. We describe how the semidefinite programming approach of \cite{candes2014towards} can be extended to  compute the solutions of the total variation minimization problem in our radial lines setting. Although this work is concerned only with the recovery of discrete measures, our framework should be readily extendible for the study of generalized total variation minimization problems (as discussed in Section \ref{sec:generalizations}. For future work, one could consider the extension of our results to the recovery of more general functions. Such results would be of interest due to the practicality of radial lines sampling for applications such as magnetic resonance imaging.

\section*{Acknowledgements}
This work was partially supported by CNRS (D\'efi Imag'in de la Misson pour l'Interdisciplinarite\'e, project CAVALIERI).
The authors would like to thank Anders C. Hansen for reading an early draft of this work and for his constructive comments.

\appendix

 \addcontentsline{toc}{section}{References}
\bibliographystyle{abbrv}
\bibliography{References}

\end{document}